\documentclass[10pt,leqno]{amsart}
\usepackage{amsfonts,amssymb}
\usepackage{color}
\usepackage{graphicx}
\newtheorem{theorem}{Theorem}[section]
\newtheorem{lemma}[theorem]{Lemma}
\newtheorem{corollary}[theorem]{Corollary}
\newtheorem{proposition}[theorem]{Proposition}
\newtheorem{assumption}[theorem]{Assumption}

\theoremstyle{definition}
\newtheorem{definition}[theorem]{Definition}

\newtheorem{remark}[theorem]{Remark}

\numberwithin{equation}{section}

\newcommand{\R}{{\mathbb R}}
\newcommand{\C}{{\mathbb C}}
\newcommand{\N}{{\mathbb N}}
\newcommand{\Z}{{\mathbb Z}}

\newcommand{\QQ}{{\mathcal{Q}}}
\newcommand{\cR}{{\cal R}}

\renewcommand{\Re}{\operatorname{Re}}
\renewcommand{\Im}{\,{\rm Im}\,}
\renewcommand{\d}{{\mathrm d}}

\renewcommand{\epsilon}{\varepsilon}
\newcommand{\cal}[1]{{\mathcal{#1}}}

\newcommand{\supp}{\operatorname{supp}}

\newcommand{\e}{\operatorname{e}}

\newcommand{\eps}{\varepsilon}

\renewcommand{\theenumi}{(\roman{enumi})}
\renewcommand{\labelenumi}{\theenumi}

\newcommand{\dd}{{\mathrm{d}}}
\newcommand{\cL}{{\mathcal{L}}}
\newcommand{\cT}{{\mathcal{T}}}

\begin{document}

\title[Square roots of divergence operators]{The square root problem for second
	order, divergence form operators with mixed boundary conditions on
	$L^p$}
\author{Pascal Auscher}
\address{Laboratoire de Math\'ematiques d'Orsay, Universit\'e Paris-Sud, UMR du
CNRS 8628, 91405 Orsay Cedex, France}
\email{Pascal.Auscher@math.u-psud.fr}

\author{Nadine Badr}
\address{Universit\'e de Lyon, CNRS, Universit\'e Lyon 1, Institut Camille
Jordan, 43, boulevard du 11 Novembre 1918, 69622 Villeurbanne Cedex, France}
\email{badr@math.univ-lyon1.fr}

\author{Robert Haller-Dintelmann}
\address{Technische Universit\"at Darmstadt, Fachbereich
Mathematik, Schlossgartenstr.\@ 7, 64298 Darmstadt, Germany}
\email{haller@mathematik.tu-darmstadt.de}

\author{Joachim Rehberg}
\address{Weierstrass Institute for Applied Analysis and Stochastics,
 Mohrenstr.\@ 39, 10117 Berlin, Germany}
\email{rehberg@wias-berlin.de}

\date{}

\begin{abstract}
We show that, under general conditions,
the operator $\bigl (-\nabla \cdot \mu \nabla +1\bigr )^{1/2}$ with mixed
boundary conditions provides a topological isomorphism between 
$W^{1,p}_D(\Omega)$ and $L^p(\Omega)$, for $p \in {]1,2[}$
if one presupposes that this isomorphism holds true for $p=2$. The domain
$\Omega$ is assumed to be bounded, the Dirichlet part $D$ of the boundary 
has to satisfy the well-known Ahlfors-David condition, whilst for the 
points from $\overline {\partial \Omega \setminus D}$ the existence of
bi-Lipschitzian boundary charts is required.
\end{abstract}

\maketitle
\section{Introduction}
The main purpose of this paper is to identify the domain of the square root of
a divergence form operator $-\nabla \cdot \mu \nabla + 1$ on $L^p(\Omega)$ as
a Sobolev space $W^{1,p}_D(\Omega)$ of differentiability order $1$ for $p \in
{]1,2]}$. (The subscript $D$ indicates the subspace of $W^{1,p}(\Omega)$ whose
elements vanish on the boundary part $D$.) Our focus lies on non-smooth
geometric situations in $\R^d$ for $d \ge 2$. So, we allow for mixed boundary
conditions and, additionally, deviate from the Lipschitz property of the domain
$\Omega$ in the following spirit: the boundary $\partial \Omega$ decomposes
into a closed subset $D$ (the Dirichlet part) and its complement, which may
share a common frontier within $\partial \Omega$. Concerning $D$, we only
demand that it satisfies the well-known Ahlfors-David condition (equivalently:
is a $(d-1)$-set in the sense of Jonsson/Wallin \cite[II.1]{jons}), and only
for points from the complement we demand bi-Lipschitzian charts around. As
special cases, the pure Dirichlet ($D = \partial \Omega$) and pure Neumann case
($D = \emptyset$) are also included in our considerations. Finally the
coefficient function $\mu$ is just supposed to be real, measurable, bounded and
elliptic in general, cf.\@ Assumption~\ref{assu-coeffi}. Together, this setting
should cover nearly all geometries that occur in real-world problems -- as long
as the domain does not have irregularities like cracks meeting the Neumann
boundary part $\partial \Omega \setminus D$. In particular, all boundary points
of a polyhedral $3$-manifold with boundary admit bi-Lipschitzian boundary
charts -- irrespective how 'wild' the local geometry is,
cf.\@ \cite[Thm.~3.10]{Ha/hoe/ka/re/zie}.

The identification of the domain for fractional powers of elliptic operators,
in particular that of square roots, has a long history. Concerning Kato's
square root problem -- in the Hilbert space $L^2$ -- see
e.g.\@ \cite{mcintosh}, \cite{terelst}, \cite{EHT13},
\cite{a/h/l/m/t} (here only the non-selfadjoint case is of interest). Early
efforts, devoted to the determination of domains for fractional powers in the
non-Hilbert space case seem to culminate in \cite{seeley}.
In recent years the problem has been investigated in the case of $L^p$ ($p \neq
2$) for instance in \cite{auschmem}, \cite{ausch/tcha98}, \cite{hyt},
\cite{jer/ke}, \cite{hal/re1}, \cite{ausch/tcha01}; but only the last three are
dedicated to the case of a nonsmooth $\Omega \neq \R^d$. In \cite{ausch/tcha01}
the domain is a strong Lipschitz domain and the boundary conditions are either
pure Dirichlet or pure Neumann. Our result generalizes this to a large extent
and, at the same time, gives a new proof for these special cases, using more
'global' arguments. Since, in the case of a non-symmetric coefficient function
$\mu$, for the nonsmooth constellations described above no general condition is
known that assures $( -\nabla \cdot \mu \nabla + 1 )^{1/2} : W^{1,2}_D(\Omega)
\to L^2(\Omega)$ to be an isomorphism, this is supposed as one of our
assumptions. This serves then as our starting point to show the corresponding
isomorphism property of $(-\nabla \cdot \mu \nabla + 1 )^{1/2} :
W^{1,p}_D(\Omega) \to L^p(\Omega)$ for $p \in {]1,2[}$. For the case $d=1$ this
is already known, even for all $p \in {]1, \infty[}$ and more general
coefficient functions $\mu$, cf.\@ \cite{AMN97}. So we stick to the case $d \ge
2$.

While the isomorphism property is already interesting in itself, our original
motivation comes from applications: having the isomorphism $(-\nabla \cdot \mu
\nabla + 1)^{1/2} : W^{1,p}_D(\Omega) \to L^p(\Omega)$ at hand, the adjoint
isomorphism $\bigl( (-\nabla \cdot \mu \nabla + 1)^{1/2} \bigr)^* = (-\nabla
\cdot \mu^T \nabla + 1 )^{1/2} : L^q(\Omega) \to W^{-1,q}_D(\Omega)$ allows to
carry over substantial properties of the operators $-\nabla \cdot \mu \nabla$
on the $L^p$-scale to the scale of $W^{-1,q}_D$-spaces for $q \in
{[2,\infty[}$. In particular, this concerns the $H^\infty$-calculus and maximal
parabolic regularity, see Section~\ref{sec-Cons}, which in turn is a powerful
tool for the treatment of linear and nonlinear parabolic equations, see
e.g.\@ \cite{pruess} and \cite{hal/re0}.

The paper is organized as follows: after presenting some notation and general
assumptions in Section~\ref{sec-hyp}, in Section~\ref{sec-Sob} we introduce the
Sobolev scale $W^{1,p}_D(\Omega)$, $1 \le p \le \infty$, related to mixed
boundary conditions and point out some of their properties. In
Section~\ref{sec-Opera} we define properly the elliptic operator under
consideration and collect some known facts for it. The main result on the
isomorphism property for the square root of the elliptic operator is precisely
formulated in Section~\ref{sec-mainResult}. The following sections contain
preparatory material for the proof of the main result, which is finished at the
end of Section~\ref{sec-Proof}. Some of these results have their own interest,
such as Hardy's inequality for mixed boundary conditions that is proved in
Section~\ref{sec-Hardy} and the results on real and complex interpolation for
the spaces $W_D^{1,p}(\Omega)$, $1 \le p \le \infty$, from
Section~\ref{sec-Interpolation}, so we shortly want to comment on these.

Our proof of Hardy's inequality heavily rests on two things: first one uses an 
operator that extends functions from $W^{1,p}_D(\Omega)$ to
$W_0^{1,p}(\Omega_\bullet)$, where $ \Omega_\bullet$ is a domain containing
$\Omega$. Then one is in a situation where the deep results of Ancona
\cite{Anc}, Lewis \cite{lewis} and Wannebo \cite{Wannebo}, combined with
Lehrb\"ack's \cite{juha} ingenious characterization of $p$-fatness, may be
applied.

The proof of the interpolation results, as well as other steps in the proof of
the main result, are fundamentally based on an adapted Calder\'on-Zygmund
decomposition for Sobolev functions. Such a decomposition was first introduced
in \cite{auschmem} and has also succesfully been used in \cite{TheseBadr}, see
also \cite{Badr09}. We have to modify it, since the main point here is, that
the decomposition has to respect the boundary conditions. This is accomplished
by incorporating Hardy's inequality into the controlling maximal operator. This
result, which is at the heart of our considerations, is contained in
Section~\ref{sec-CZD}.

All these preparations, together with off-diagonal estimates for the semigroup
generated by our operator, cf.\@ Section~\ref{sec-OffD}, lead to the
proof of the main result in Section~\ref{sec-Proof}. Finally, in
Section~\ref{sec-Cons} we draw some consequences, as already sketched above.

After having finished this work we got to know of the paper \cite{brewst}.
There, among other deep things, Lemma \ref{l-extend} and the interpolation 
results of Section~\ref{sec-Interpolation} are also proved -- and this in an
even much broader setting than ours.

\subsection*{Acknowledgments}
In 2012, after we asked him a question, V.~Maz'ya proposed a proof of
Proposition~\ref{p-juha} that heavily relied on several deep results from his
book \cite{Mazya}. Actually there was an earlier reference in the literature
with a different approach that, provided a simple lemma is established, applies
directly. It was again V.~Maz'ya who drew our attention to the fact that
something like this lemma is needed. We warmly thank him for all that.

The authors also want to thank A.~Ancona, M.~Egert, P.~Koskela, and
W.~Tre\-bels for valuable discussions and hints on the topic.

Finally we thank the referees for many valuable hints.

%
%
\section{Notation and general assumptions} \label{sec-hyp}
%
%
%
\noindent
Throughout the paper  we will use $\mathrm{x}, \mathrm{y},\dots$ for vectors in
$\R^d$ and the symbol $B(\mathrm x,r)$ stands for the ball in $\R^d$ around
$\mathrm x$ with radius $r$. For $E,F \subseteq \R^d$ we denote by $\dd(E,F)$
the distance between $E$ and $F$, and if $E = \{\mathrm x\}$, then we write
$\dd(\mathrm x,F)$ or $\dd_F(\mathrm{x})$ instead.

Regarding our geometric setting, we suppose the following assumption throughout
this work.
\begin{assumption} \label{assu-general}
 \begin{enumerate}
 \item \label{assu-general:i} Let $d \ge 2$, let $\Omega \subseteq \R^d$ be
	a bounded domain and let $D$ be a closed subset of the boundary
	$\partial \Omega$ (to be understood as the Dirichlet boundary part).
	For every $\mathrm x \in \overline{\partial \Omega \setminus D}$ there
	exists an open neighbourhood $U_{\mathrm x}$ of $\mathrm x$ and a bi-Lipschitz map
	$\phi_{\mathrm x}$ from $U_{\mathrm x}$ onto the cube $K :=
	{]{-1},1[}^d$, such that the following three conditions are satisfied:
	\begin{align*}
	  \phi_{\mathrm x}(\mathrm x) &= 0, \nonumber \\
	  \phi_{\mathrm x}(U_{\mathrm x} \cap \Omega) &= \{ \mathrm x \in
		K : x_d < 0 \} =: K_-,
		\\
	  \phi_{\mathrm x}(U_{\mathrm x} \cap \partial \Omega) &= \{
		\mathrm x \in K : x_d = 0 \} =: \Sigma. \nonumber
	\end{align*}
 \item  We suppose that $D$ is either empty or
	satisfies the \emph{Ahlfors-David condition}:
	There are constants $c_0, c_1 > 0$ and $r_{AD} > 0$, such that for all
	$\mathrm x \in D$ and all $r \in {]0,r_{AD}]}$
	\begin{equation} \label{e-ahlf}
	  c_0 r^{d-1} \le \mathcal H_{d-1} (D \cap B(\mathrm x,r) ) \le c_1
		r^{d-1},
	\end{equation}
	where $\mathcal H_{d-1}$ denotes (here and in the sequel) the
	$(d-1)$-dimensional Hausdorff measure, defined by
	\[ \mathcal H_{d-1}(A) := \liminf_{\eps \to 0} \Bigl
		\{\sum_{j=1}^\infty \mathrm{diam}(A_j)^{d-1} : \ A_j \subseteq
		\R^d, \ \mathrm{diam}(A_j) \le \eps, \ A \subseteq
		\bigcup_{j=1}^\infty A_j \Bigr\}.
	\]
 \end{enumerate}
\end{assumption}

\begin{remark} \label{r-surfmeas}
 \begin{enumerate}
 \item Condition~\eqref{e-ahlf} means that $D$ is a $(d-1)$-set in the sense 
	of Jonsson/Wallin \cite[Ch.~II]{jons}.
 \item \label{r-surfmeas:ii} On the set $\partial \Omega \cap \bigl(
	\bigcup_{\mathrm x \in  \partial \Omega \setminus D} U_\mathrm x
	\bigr)$ the measure $\mathcal H_{d-1}$ equals the surface measure
	$\sigma$ which can be constructed via the bi-Lipschitzian charts
	$\phi_\mathrm x$ around these boundary points, compare
	\cite[Section~3.3.4~C]{ev/gar} or \cite[Section~3]{coerciv}. In
	particular, \eqref{e-ahlf} assures the property $\sigma \bigl( D \cap
	\bigl( \cup _{\mathrm x \in \partial \Omega \setminus D} U_\mathrm x
	\bigr) \bigr) > 0$.
\item We emphasize that the cases $D = \partial \Omega$ or $D =	\emptyset$ are
	not excluded.
 \end{enumerate}
\end{remark}


If $B$ is a closed operator on a Banach space $X$, then we denote by
$\mathrm{dom}_X(B)$ the domain of this operator. $\mathcal L(X,Y)$ denotes the
space of linear, continuous operators from $X$ into $Y$; if $X = Y$, then we
abbreviate $\mathcal L(X)$. Furthermore, we will write $\langle \cdot, \cdot
\rangle_{X'}$ for the pairing of elements of $X$ and the dual space $X'$ of
$X$.

Finally, the letters $c$ and $C$ denote generic constants that may change value
from occurrence to occurrence.

%
%
%
%
\section{Sobolev spaces related to boundary conditions} \label{sec-Sob}
%
%
%
%
\noindent
In this section we will introduce the Sobolev spaces related to mixed boundary
conditions and prove some results related to them that will be needed later.

If $\Upsilon$ is an open subset of $\R^d$ and $F$ a closed subset of
$\overline \Upsilon$, e.g.\@ the Dirichlet part $D$ of $\partial \Omega$, then
for $1 \le q < \infty$ we define $W^{1,q}_F(\Upsilon)$ as the completion
of
\begin{equation} \label{e-defC}
  C^\infty_F(\Upsilon) := \{ \psi|_\Upsilon : \psi \in C_0^\infty(\R^d),
	\; \supp(\psi) \cap F = \emptyset \}
\end{equation}
with respect to the norm $\psi \mapsto \bigl( \int_\Upsilon |\nabla \psi|^q +
|\psi|^q \; \d \mathrm x \bigr)^{1/q}$. For $1 < q < \infty$ the dual of
this space will be denoted by $W^{-1,q'}_F(\Upsilon)$ with $1/q + 1/q' = 1$.
Here, the dual is to be understood with respect to the extended $L^2$ scalar
product, or, in other words: $W^{-1,q'}_F(\Upsilon)$ is the space of continuous
antilinear forms on $W^{1,q}_F(\Upsilon)$.


Finally, we define the respective spaces for the case $q = \infty$. We set
$W^{1,\infty}_F(\Upsilon):= \mathrm{Lip}_{\infty, F}(\Upsilon)$ with
\begin{equation} \label{e-LipRaum}
 \mathrm{Lip}_{\infty, F}(\Upsilon) := \bigl\{ f|_\Upsilon : f \in
	(L^\infty \cap \mathrm{Lip})(\R^d), f|_F = 0 \bigr\} = \bigl\{ f \in
	(L^\infty \cap \mathrm{Lip})(\Upsilon), f|_F = 0 \bigr\}.
\end{equation}
The norm on this space is 
\[ \|f\|_{L^\infty(\Upsilon)} + \sup_{\mathrm{x}, \mathrm{y} \in \Upsilon,
	\mathrm{x} \ne \mathrm{y}}
	\frac{|f(\mathrm{x}) - f(\mathrm{y})|}{|\mathrm{x} - \mathrm{y}|}.
\]
The last equality in \eqref{e-LipRaum} is a consequence of the Whitney
extension theorem. We have $\mathrm{Lip}_{\infty, F}(\Upsilon) \subseteq
\bigl\{ f \in \mathcal{W}^{1,\infty}(\Upsilon) : f|_F = 0 \bigr\}$
($\mathcal{W}^{1,\infty}(\Upsilon)$ is defined using distributions) and the
converse holds iff $\Omega $ is uniformly locally convex by \cite[Theorem 7]{hajlasz3}.

In order to simplify notation, we drop the $\Omega$ in the notation of spaces,
if misunderstandings are not to be expected. Thus, function spaces without an
explicitely given domain are to be understood as function spaces on $\Omega$.

\begin{lemma}\label{inclusion}
 Let $\Upsilon \subseteq \R^d$ be a bounded domain and $F$ a (relatively)
 closed subset of $\partial \Upsilon$. Then $W^{1,\infty}_F(\Upsilon) \subseteq
 W^{1,q}_F(\Upsilon)$ for $1 \le q <\infty$.
\end{lemma}

\begin{proof}
 Let $(\alpha_n)_n$ be the sequence of cut-off functions defined on $\R^+$ by
 \[ \alpha_n(t)=  \begin{cases}
	0, & \text{if } 0 \le t < 1/n, \\
	n t - 1, & \text{if } 1/n \leq t \leq 2/n, \\
	1, & \text{if } t > 2/n.
	\end{cases}
 \]
 Remark that for $t \ne 0$ the sequence $\alpha_{n}(t)$ tends to $1$ as $n \to
 \infty$. Furthermore, for all $t \ge 0$ we have $0 \le t \alpha_n'(t)\le 2$
 and the sequence $(t \alpha_n'(t))_n$ tends to $0$.

 For $\mathrm{x} \in \R^d$ we set $w_n(\mathrm{x}) :=
 \alpha_n(\dd(\mathrm{x},F))$. Then, by the above considerations, $w_n \to 1$ 
 almost everywhere as $n \to \infty$.
 The function $d(\cdot,F)$ is Lipschitzian with Lipschitz constant $1$, hence
 it belongs to $W^{1,\infty}_{\mathrm{loc}}(\R^d)$,
 cf.\@ \cite[Ch.~4.2.3 Thm.~5]{ev/gar}. Since $\alpha$ is piecewise smooth, the
 usual chain rule for weak differentiation (cf.\@ \cite[Ch.~7.4 Thm.~7.8]{Gil})
 applies, which gives
 \[  |\nabla w_n(\mathrm{x})| = \bigl| \alpha'_n(\dd(\mathrm{x}, F)) \bigr|
	|\nabla \dd(\mathrm{x}, F)| \le  \bigl| \alpha'_n(\dd(\mathrm{x}, F))|
 \]
 almost everywhere on $\R^d$. Thus $\dd(\mathrm{x}, F)
 | \nabla w_n(\mathrm{x})|$ is bounded and converges to $0$ almost everywhere
 as $n \to \infty$.

 Let $g\in W^{1,\infty}_F(\Upsilon)$, which we consider as defined on $\R^d$.
 Since $\Upsilon$ is bounded, we may assume that $g$ has compact support in
 some large ball $B$. Let $g_n := g w_n$. Then $g_n$ is compactly supported in
 $B$ and  in $\R^d \setminus F$. We claim that  $g_n \to g$ in $W^{1,q}(\R^d)$.
 Indeed, $g - g_n = g (1 - w_n)$ and, by the dominated convergence theorem,
 $g (1 - w_n) \to 0$ in $L^q(\R^d)$, since $w_n \to 1$.

 Now, for the gradient, we have
 \[ \nabla g_n - \nabla g = (w_n - 1) \nabla g + g \nabla w_n.
 \]
 Again by the dominated convergence theorem, the first term converges to $0$ in
 $L^q(\R^d)$.

 It remains to prove that  $\|g \nabla w_n\|_{L^q(\R^d)}$ converges to $0$. We
 have
 \begin{equation} \label{e-Lip}
  (g \nabla w_n)(\mathrm{x}) = 
\begin{cases}
0, \quad & \text {if } \mathrm{x} \in F,\\
\frac{g(\mathrm{x})}{\dd(\mathrm{x}, F)}
	\dd(\mathrm{x}, F) \nabla w_n(\mathrm{x})& \text{a.e. on } \R^d 
\setminus F.
\end{cases} 
 \end{equation}
 Since $g$ is Lipschitz continuous on the whole of $\R^d$ and satisfies $g = 0$
 on $F$, we find 
 \[ \sup_{\mathrm{x} \in \R^d} \Bigl| \frac{g(\mathrm{x})}{\dd(\mathrm{x}, F)}
	\Bigr| = \sup_{\mathrm{x} \in \R^d} \Bigl|
	\frac{g(\mathrm{x}) - g(\mathrm{x}_*)}{\mathrm{x} - \mathrm{x}_*}
	\Bigr| \le C,
 \]
 where $\mathrm{x}_* \in F$ denotes an element of $F$ that realizes the
 distance of $\mathrm{x}$ to $F$. So both factors on the right hand side in
 \eqref{e-Lip} are bounded and $\dd(\mathrm{x}, F) \nabla w_n(\mathrm{x})$ goes
 to $0$ almost everywhere as $n \to \infty$. Thus, since $g$ has compact
 support, the dominated convergence theorem yields $g \nabla w_n \to 0$ in
 $L^q(\R^d)$.

 Finally, it suffices to convolve this approximation with a smooth mollifying
 function that has small support to conclude $g \in W^{1,q}_{F}(\Upsilon)$.
\end{proof}

Next, we establish the following extension property for function spaces on
domains, satisfying just part \ref{assu-general:i} of
Assumption~\ref{assu-general}. This has been proved in \cite{tomjo} for
$q = 2$. For convenience of the reader we include a proof.
%
\begin{lemma} \label{l-extend}
 Let $\Omega$ and $D$ satisfy
 Assumption~\ref{assu-general}~\ref{assu-general:i}. Then there is a continuous
 extension operator $\mathfrak E$ which maps each space $W^{1,q}_D(\Omega)$
 continuously into $W^{1,q}_D(\R^d)$, $q \in [1,\infty]$. Moreover,
 $\mathfrak E$ maps $L^q(\Omega)$ continuously into $L^q(\R^d)$ for $q \in
 [1, \infty]$.
\end{lemma}
\begin{proof}
 For every $\mathrm x \in \overline{\partial \Omega \setminus D}$ let the set
 $U_\mathrm x$ be an open neighbourhood that satisfies the condition from
 Assumption~\ref{assu-general}~\ref{assu-general:i}. Let $U_{\mathrm x_1},
 \ldots, U_{\mathrm x_\ell}$ be a finite subcovering of
 $\overline{\partial \Omega \setminus D}$ and let $\eta \in C^\infty_0(\R^d)$
 be a function that is identically one in a neighbourhood of
 $\overline{\partial \Omega \setminus D}$ and has its support in $U :=
 \bigcup_{j=1}^\ell U_{\mathrm x_j}$.

 Assume $\psi \in C^\infty_D(\Omega)$; then we can write $\psi = \eta \psi +
 (1 - \eta) \psi$. By the definition of $C^\infty_D(\Omega)$ and $\eta$ it is
 clear that the support of $(1 - \eta) \psi$ is contained in $\Omega$, thus
 this function may be extended by $0$ to the whole space $\R^d$ -- while its
 $W^{1,q}$-norm is preserved.

 It remains to define the extension of the function $\eta \psi$, what we will
 do now. For this, let $\eta_1, \ldots, \eta_\ell$ be a partition of unity on
 $\supp(\eta)$, subordinated to the covering $U_{\mathrm x_1}, \ldots,
 U_{\mathrm x_\ell}$. Then we can write $\eta \psi = \sum_{r=1}^\ell \eta_r
 \eta \psi$ and have to define an extension for every function $\eta_r \eta
 \psi$. For doing so, we first transform the corresponding function under the
 corresponding mapping $\phi_{\mathrm x_r}$ from
 Assumption~\ref{assu-general}~\ref{assu-general:i} to
 $\widetilde{\eta_r \eta \psi} = (\eta_r \eta \psi) \circ
 \phi_{\mathrm x_r}^{-1}$ on the half cube $K_-$. Afterwards, by even
 reflection, one obtains a function $\widehat{\eta_r \eta \psi} \in W^{1,q}(K)$
 on the cube $K$. It is clear by construction that
 $\supp(\widehat{\eta_r \eta \psi})$ has a positive distance to $\partial K$.
 Transforming back, one ends up with a function $\underline{\eta_r \eta \psi}
 \in W^{1,q}(U_{\mathrm x_r})$ whose support has a positive distance to
 $\partial U_{\mathrm x_r}$. Thus, this function may also be extended by $0$ to
 the whole of $\R^d$, preserving again the $W^{1,q}$ norm.

 Lastly, one observes that all the mappings $W^{1,q}(U_{\mathrm x_r} \cap
 \Omega) \ni \eta_r \eta \psi \mapsto \widetilde{\eta_r \eta \psi} \in
 W^{1,q}(K_-)$, $W^{1,q}(K_-) \ni \widetilde{\eta_r \eta \psi} \mapsto
 \widehat{\eta_r \eta \psi} \in W^{1,q}(K)$ and $W^{1,q}(K) \ni
 \widehat{\eta_r \eta \psi} \mapsto \underline{\eta_r \eta \psi} \in
 W^{1,q}(U_{\mathrm x_r})$ are continuous. Thus, adding up, one arrives at an
 extension of $\psi$ whose $W^{1,q}(\R^d)$-norm may be estimated by $c
 \|\psi\|_{W^{1,q}(\Omega)}$ with $c$ independent from $\psi$. Hence, the
 mapping $\mathfrak E$, up to now defined on $C^\infty_D(\Omega)$, continuously
 and uniquely extends to a mapping from $W^{1,q}_D$ to $W^{1,q}(\R^d)$.

 It remains to show that the images in fact even are in $W^{1,q}_D(\R^d)$. For
 doing so, one first observes that, by construction of the extension operator,
 for any $\psi \in C^\infty_D(\Omega)$, the support of the extended function
 $\mathfrak E \psi$ has a positive distance to $D$ -- but $\mathfrak E \psi$
 need not be smooth. Clearly, one may convolve $\mathfrak E \psi$ suitably in
 order to obtain an appropriate approximation in the $W^{1,q}(\R^d)$-norm --
 maintaining a positive distance of the support to the set $D$. Thus,
 $\mathfrak E$ maps $C^\infty_D(\Omega)$ continuously into $W^{1,q}_D(\R^d)$,
 what is also true for its continuous extension to the whole space
 $W^{1,q}_D(\Omega)$.

 It is not hard to see that the operator $\mathfrak E$ extends to a continuous
 operator from $L^q(\Omega)$ to $L^q(\R^d)$, where $q \in [1,\infty]$.
\end{proof}

\begin{remark} \label{r-remain}
\begin{enumerate}
\item \label{r-remain:i} By construction, all extended functions $\mathfrak E
	f$ have their support in $\Omega \cup \bigcup_{j=1}^\ell
	U_{\mathrm x_j}$, and, hence, in a suitably large ball.
\item \label{r-remain:ii} Employing Lemma~\ref{l-extend} in conjunction with
	\ref{r-remain:i}, one can establish the corresponding Sobolev
	embeddings of $W^{1,p}_D(\Omega)$ into the appropriate $L^q$-spaces
	(compactness included) in a
	straightforward manner.
\item  When combining $\mathfrak E$ with a multiplication operator that is
	induced by a function $\eta_0 \in C_0^\infty(\R^d)$, $\eta_0 \equiv 1$
	on $\Omega$, one may achieve that the support of the extended functions
	shrinks to a set which is arbitrarily close to $\Omega$.
\item It is not hard to see that functions from $W^{1,p}_D(\Omega)$ admit a
	trace on the set $\partial \Omega \setminus D$, thanks to the
	bi-Lipschitz charts presumed in our general
	Assumption~\ref{assu-general}. Moreover, the Jonsson/Wallin results in
	\cite[Ch.~VII]{jons} show that the extended functions $\mathfrak E f$
	admit a trace on the set $D$. A much more delicate point is the
	existence of a trace on $D$ and the coincidence with the trace of the
	extended function. This question is deeply investigated in
	\cite[Ch.~5]{brewst}, cf.\@ Theorem~5.2 and Corollary~5.3, compare also
	\cite[Ch. VIII Prop. 2]{jons}.

	In the following these subtle considerations will not be needed.
\end{enumerate}
\end{remark}

\begin{remark} \label{r-equivnorms}
 The geometric setting of Assumption~\ref{assu-general} still allows for a
 Poincar\'e inequality for functions from $W^{1,p}_D$, as soon as $D \neq
 \emptyset$. This is proved in \cite[Thm.~3.5]{coerciv}, if $\Omega$ is a
 Lipschitz domain. In fact, the proof only needs that a part of $D$ admits
 positive boundary measure and this is guaranteed by
 Remark~\ref{r-surfmeas}~\ref{r-surfmeas:ii}.

 This Poincar\'e inequality entails that, whenever $D \neq \emptyset$, the
 norms given by $\| f \|_{W^{1,p}_D}$ and $\| \nabla f \|_{L^p}$ for $f \in
 W^{1,p}_D$ are equivalent. So, in this case, in all subsequent considerations
 one may freely replace the one by the other.
\end{remark}

\section{The divergence operator: Definition and elementary properties}
	\label{sec-Opera}

\noindent
We now turn to the definition of the elliptic divergence form operator that
will be investigated. Let us first introduce the ellipticity supposition on the
coefficients.
\begin{assumption} \label{assu-coeff}
 The coefficient function $\mu$ is a Lebesgue measurable, bounded function on
 $\Omega$ taking its values in the set of real, $d \times d$ matrices,
 satisfying for some $\mu_\bullet > 0$ the usual ellipticity condition
 \[ \xi^T \mu(\mathrm x) \xi \ge \mu_\bullet |\xi|^2, \qquad \text{for all }
	 	\xi \in \R^d \text{ and almost all } \mathrm x \in \Omega.
 \]
\end{assumption}
The operator $A : W^{1,2}_D \to W^{-1,2}_D$ is defined by
\[ \langle A\psi, \varphi\rangle_{W^{-1,2}_D} := \mathfrak t(\psi, \varphi) :=
	\int_\Omega \mu \nabla \psi \cdot \nabla \overline \varphi
	\; \dd \mathrm{x}, \quad \psi, \varphi \in W^{1,2}_D.
\]
Often we will write more suggestively $-\nabla \cdot \mu \nabla$ instead of
$A$.

The $L^2$ realization of $A$, i.e.\@ the maximal restriction of $A$ to the
space $L^2$, will be denoted by the same symbol $A$; clearly this is identical
with the operator that is induced by the sesquilinear form $\mathfrak t$. If
$B$ is a densely defined, closed operator on $L^2$, then by the $L^p$
realization of $B$ we mean its restriction to $L^p$ if $p > 2$ and the
$L^p$~closure of $B$ if $p \in {[1, 2 [}$. (For all operators we have in mind,
this $L^p$-closure exists.)

As a starting point of our considerations we assume that the square root of
our operator is well-behaved on $L^2$. 

\begin{assumption} \label{assu-coeffi}
 The operator $(-\nabla \cdot \mu \nabla + 1 )^{1/2} : W^{1,2}_D \to L^2$
 provides a  topological isomorphism; in other words: the domain of  $(-\nabla
 \cdot \mu \nabla + 1 )^{1/2}$ on $L^2$ is the form domain $W^{1,2}_D$.
\end{assumption}

\begin{remark} \label{r-valid}
 By a recent result in \cite{EHT13} the isomorphism property which is assumed
 in the above assumption is known in our context under the additional
 hypotheses that $\Omega$ is a $d$-set, i.e.\@ there is a constant $c > 0$,
 such that
 \[ \frac 1c r^d \le \mathcal H_d \bigl( \Omega \cap B(\mathrm x,r) \bigr) \le
	c r^d \qquad \text{for all } \mathrm x \in \Omega \text{ and } r \in
	[0,1],
 \]
 where $\mathcal H_d$ denotes the $d$-dimensional Hausdorff measure.
 Furthermore, some other remarkable special cases in this context are
 available:
 \begin{enumerate} 
  \item If this assumption is satisfied for a coefficient function $\mu$, then
	it is also true for the adjoint coefficient function,
	cf.\@ \cite[Thm.~8.2]{Ouh05}.
  \item Assumption~\ref{assu-coeffi} is always fulfilled if the coefficient
	function $\mu$ takes its values in the set of real \emph{symmetric}
	$d\times d$-matrices.
  \item In view of non-symmetric coefficient functions see \cite{mcintosh} and
	\cite{terelst}.
 \end{enumerate}
\end{remark}

Finally, we collect some facts on $- \nabla \cdot \mu \nabla$ as an operator on
the $L^2$ and on the $L^p$ scale.
%
\begin{proposition} \label{p-basicl2}
 Let $\Omega \subseteq \R^d$ be a domain and let $D \subseteq \partial \Omega$
 (relatively) closed.
 \begin{enumerate}
  \item \label{p-basicl2:i} The restriction of $-\nabla \cdot \mu \nabla$ to
	$L^2$ is a densely defined sectorial operator.
  \item The operator $\nabla \cdot \mu \nabla$ generates an analytic
	semigroup on $L^2$.
  \item \label{p-basicl2:iii}The form domain $W^{1,2}_D$ is invariant under
	multiplication with functions from $W^{1,q}$, if $q > d$.
 \end{enumerate}
\end{proposition}
%
\begin{proof}
 \begin{enumerate}
  \item It is not hard to see that the form $\mathfrak t$ is closed and its
	numerical range lies in the sector $\{z \in \C: |\Im z| \le
	\frac{\|\mu\|_{L^\infty}}{\mu_\bullet} \Re z \}$. Thus, the assertion
	follows from a classical representation theorem for forms, see
	\cite[Ch.~VI.2.1]{kato}.
  \item This follows from \ref{p-basicl2:i} and \cite[Ch.~V.3.2]{kato}.
  \item First, for $u \in C^\infty_D(\Omega)$ and $v \in C^\infty(\Omega)$ the
	product $u v$ is obviously in $C^\infty_D(\Omega) \subseteq W^{1,2}_D$.
	But, by definition of $W^{1,2}_D$, the set $C^\infty_D(\Omega)$ (see
	\eqref{e-defC}) is dense in $W^{1,2}_D$ and $C^\infty(\Omega)$ is dense
	in $W^{1,q}$. Thus, the assertion is implied by the continuity of the
	mapping
	\[ W^{1,2}_D \times W^{1,q} \ni (u,v) \mapsto uv \in W^{1,2},
	\]
	because $W^{1,2}_D$ is closed in $W^{1,2}$.
	\qedhere
 \end{enumerate}
\end{proof}
%

%
\begin{proposition} \label{t-gausss}
Let $\Omega$ and $D$ satisfy Assumption~\ref{assu-general}~\ref{assu-general:i}.
 Then the semigroup generated by $\nabla \cdot \mu \nabla$ in
$L^2$ satisfies upper Gaussian estimates, precisely:
\[ (\e^{t \nabla \cdot \mu \nabla} f)(\mathrm x) = \int_\Omega
	K_t(\mathrm x, \mathrm y) f(\mathrm y) \; \dd \mathrm y, \quad
	\text{ for a.a. } \mathrm x \in \Omega, \; f \in L^2,
\]
for some measurable function $K_t : \Omega \times \Omega \to \R_+$, and for all
$\epsilon > 0$ there exist constants $C, c > 0$,
such that
\begin{equation}\label{5.5}
  0 \leq K_t(\mathrm x, \mathrm y) \le \frac{C}{t^{d/2}} \;
        \e^{- c \frac{|\mathrm x - \mathrm y|^2}{t}} \e^{\epsilon t}, \quad t
	> 0, \; a.a. \; \mathrm x, \mathrm y
        \in \Omega.
\end{equation}
\end{proposition}
%
\begin{proof}
 A proof is given in \cite{tomjo} -- heavily resting on \cite{arel}, compare
 also \cite[Thm.~6.10]{Ouh05}.
\end{proof}
\begin{proposition} \label{p-semigr}
Let $\Omega$ and $D$ satisfy Assumption~\ref{assu-general}~\ref{assu-general:i}. 
 \begin{enumerate} 
  \item For every $p \in [1,\infty]$, the operator $\nabla \cdot \mu \nabla$
	generates a semigroup of contractions on $L^p$.
  \item \label{p-semigr:ii} For all $q \in {]1, \infty[}$ the operator
	$-\nabla \cdot \mu \nabla + 1$ admits a bounded $H^\infty$-calculus on
	$L^q$ with $H^\infty$-angle $\arctan
	\frac{\|\mu\|_{L^\infty}}{\mu_\bullet}$. In particular, it admits
	bounded imaginary powers.
 \end{enumerate}
\end{proposition}
\begin{proof}
 \begin{enumerate}
  \item The operator $\nabla \cdot \mu \nabla$ generates a semigroup of
	contractions on $L^2$ (see \cite[Thm~1.54]{Ouh05}) as well as on
	$L^\infty$ (see \cite[Ch.~4.3.1]{Ouh05}). By interpolation this carries
	over to every $L^q$ with $q \in {]2,\infty[}$ and, by duality, to $q
	\in [1,2]$.
  \item Since the numerical range of $-\nabla \cdot \mu \nabla$ is contained in
	the sector $\{z \in \C : |\Im z| \le
	\frac{\|\mu\|_{L^\infty}}{\mu_\bullet} \Re z \}$, the assertion holds
	true for $q = 2$, see \cite[Cor.~7.1.17]{Haa06}. Secondly, the
	semigroup generated by $\nabla \cdot \mu \nabla - 1$ obeys the Gaussian
	estimate~\eqref{5.5} with $\epsilon = 0$. Thus, the first assertion
	follows from \cite[Theorem~3.1]{duong/robin}. The second claim is a
	consequence of the first, see \cite[Section~2.4]{DHP03}.
  \qedhere
 \end{enumerate}
\end{proof}


%
%
\section{The main result: the isomorphism property of the square root}
\label{sec-mainResult}
%
%
%
%
\noindent
We can now formulate our main goal, that is to prove that the mapping
\[ ( A + 1)^{1/2} = ( -\nabla \cdot \mu \nabla+1 )^{1/2} : W^{1,q}_D \to L^q
\]
is a topological isomorphism for $q \in \left] 1, 2 \right[$. We abbreviate
$-\nabla \cdot \mu \nabla + 1$ by $A_0$ throughout the rest of this work.

More precisely, we want to show the following main result of this
paper.
%
\begin{theorem} \label{t-mainsect}
 Under Assumptions~\ref{assu-general}, \ref{assu-coeff} and \ref{assu-coeffi}
 the following holds true:
 \begin{enumerate}
  \item \label{t-mainsect:i} For every $q \in {]1,2]}$ the operator
	$A_0^{-1/2}$ is a continuous operator from $L^q$ into $W^{1,q}_D$.
	Hence, its adjoint continuously maps $W^{-1,q}_D$ into $L^q$ for any $q
	\in {[2,\infty[}$.
\item \label{t-mainsect:ii} Moreover, if $q \in {]1,2]}$, then $A_0^{1/2}$ maps 
        $W^{1,q}_D$ continuously into $L^q$. Hence, its adjoint continuously 
        maps $L^q$ into $W^{-1,q}_D$ for any $q \in {[2,\infty[}$.
\end{enumerate}
\end{theorem}
%
We can immediately give the proof of \ref{t-mainsect:i}, i.e.\@ the continuity
of the operator $A_0^{-1/2} : L^q \to W^{1,q}_D$. We observe that this follows,
whenever
\begin{itemize}
\item[1.] The Riesz transform $\nabla A_0^{- 1/2}$ is a bounded operator on
	$L^q$, and, additionally,
\item[2.] $A_0^{-1/2}$ maps $L^q$ into $W^{1,q}_D$.
\end{itemize}
The first item is proved in \cite[Thm.~7.26]{Ouh05}, compare also \cite{DM99}.
It remains to show 2. The first point makes clear that $A_0^{-1/2}$ maps $L^q$
continuously into $W^{1,q}$, thus one only has to verify the correct boundary
behavior of the images. If $f \in L^2 \hookrightarrow L^q$, then one has
$A_0^{-1/2}f \in W^{1,2}_D \hookrightarrow W^{1,q}_D$, due to
Assumption~\ref{assu-coeffi}. Thus, the assertion follows from 1. and the
density of $L^2$ in $L^q$. 
%
\begin{remark} \label{r-nottrue}
Theorem~\ref{t-mainsect}~\ref{t-mainsect:i} is not true for other values of
$q$ in general, see \cite[Ch.~4]{auschmem} for a further discussion.
\end{remark}
%

The hard work is to prove the second part, that is the continuity of
$A_0^{1/2} : W^{1,q}_D \to L^q$. The proof is inspired by \cite{auschmem},
where this is shown in the case $\Omega = \R^d$, and will be developed in the
following five sections.

%
%
%
%

\section{Hardy's inequality} \label{sec-Hardy}
%
%
%
%
\noindent
A major tool in our considerations is an inequality of Hardy type for functions
in $W^{1,p}_D$, so functions that vanish only on the part $D$ of the boundary.

We recall that, for a set $F \subseteq \R^d$, the symbol $\d_F$ denotes the
function on $\R^d$ that measures the distance to $F$. The result we want to
show in this section, is the following.
%
\begin{theorem} \label{t-hardy}
 Under Assumption~\ref{assu-general}, for every
 $p \in {]1,\infty[}$ there is a constant $c_p$, such that
 \[ \int_{\Omega} \left | \frac{f}{\d_D} \right |^p
	\, \dd \mathrm{x} \; \leq \; c_p \int_{\Omega} |\nabla f|^p
	\; \dd \mathrm{x}
 \]
 holds for all $f \in W^{1,p}_D$.
\end{theorem}
%
Since the statement of this theorem is void for $D = \emptyset$, we exclude
that case for this entire section. Please note, that then the norm on the
spaces $W^{1,p}_D$ may be taken as $\| \nabla \cdot \|_p$ in view of the
Ahlfors-David condition of $D$.

Let us first quote the two deep results on which the proof of
Theorem~\ref{t-hardy} will base.

\begin{proposition}[see \cite{lewis}, \cite{Wannebo}, see also \cite{korte}]
 \label{p-lewis}
 Let $\Xi \subseteq \R^d$ be a domain whose complement $K := \R^d \setminus
 \Xi$ is uniformly \emph {$p$-fat} (cf.\@ \cite{lewis} or \cite{korte}). Then
 Hardy's inequality
 \begin{equation} \label{e-harrd}
  \int_\Xi \left| \frac {g}{\d_K} \right|^p \; \d \mathrm x = \int_\Xi
	\left| \frac{g}{\d_{\partial \Xi}} \right|^p \; \d \mathrm x \le c
	\int_\Xi |\nabla g|^p \; \d \mathrm x 
\end{equation}
 holds for all $g \in C^\infty_0(\Xi)$ (and extends to all $g \in
 W^{1,p}_0(\Xi)$, $p \in {]1,\infty[}$ by density).
\end{proposition}
 
\begin{proposition}[\protect{\cite[Theorem 1]{juha}}] \label{p-juha}
 Let $\Xi \subseteq \R^d$ be a domain and let $\mathcal H^\infty_{d-1}$
 denote the $(d-1)$-dimensional Hausdorff content, i.e.
 \[ \mathcal H^\infty_{d-1}(A) := \inf \Bigl\{ \sum_{j=1}^\infty r_j^{d-1} :
	\ x_j \in A, \ r_j > 0, \ A \subseteq \bigcup_{j=1}^\infty B(x_j, r_j)
	\Bigr\}.
 \]
 If $\Xi$ satisfies the \emph{inner
 boundary density condition}, i.e.
 \begin{equation} \label{e-fatn}
  \mathcal H_{d-1}^\infty \bigl( \partial \Xi \cap
	B(\mathrm x, 2 \d_{\partial \Xi}(\mathrm x)) \bigr) \ge c
	\, \d_{\partial \Xi}(\mathrm x)^{d-1}, \quad \mathrm x \in \Xi,
 \end{equation}
 for some constant $c > 0$, then the complement of $\Xi$ in $\R^d$ is uniformly
 $p$-fat for all $p \in {]1,\infty[}$.
\end{proposition}
The subsequent lemma will serve as the instrument to reduce our case to the
situation of a pure Dirichlet boundary.

\begin{lemma} \label{l-1eck}
 Let $B\supseteq \overline \Omega$ be an open ball. We define $\Omega_\bullet$
 as the union of all open, connected subsets of $B$ that contain $\Omega$ and
 avoid $D$. Then $\Omega_\bullet$ is open and connected and we have $\partial
 \Omega_\bullet = D$ or $\partial \Omega_\bullet = D \cup
 \partial B$.
\end{lemma}

\begin{proof}
 The first assertion is obvious. The connectedness follows from the fact that
 all the sets that, by forming their union, generate $\Omega_\bullet$ contain
 $\Omega$, and, hence, a common point. It remains to show the last assertion.
 Clearly, we have $\partial \Omega_\bullet \subseteq  \overline B$.

 We claim that $D \subseteq \partial \Omega_\bullet$: Let $\mathrm x \in  D$.
 As $D \subseteq \partial \Omega$, we know that $\mathrm x$ is an accumulation
 point of $\Omega$ and thus also of $\Omega_\bullet$, since $\Omega \subseteq
 \Omega_\bullet$. Furthermore $\mathrm x \not \in \Omega_\bullet$. Hence,
 $\mathrm x \in \partial \Omega_\bullet$.

 We claim that $\partial \Omega_\bullet \subseteq \partial B \cup D$. Assume
 not. Then there exists $\mathrm x \in \partial \Omega_\bullet$ with $\mathrm x
 \in B \setminus D$. As $B \setminus D$ is open, it contains an open ball
 $K_\mathrm x$ centred at $\mathrm x$. Then $\Omega_\bullet \cup K_\mathrm x$
 is an open and connected (since $\mathrm x$ is a point of accumulation of
 $\Omega_\bullet$, the set $\Omega_\bullet \cap K_\mathrm x$ is not empty) set
 containing $\Omega$, contained in $B$ and not meeting $D$. As it strictly
 contains $\Omega_\bullet$, this contradicts the definition of
 $\Omega_\bullet$.

 Let us now consider an annulus $K_B \subseteq B$ that is adjacent to
 $\partial B$ and does not intersect $\overline \Omega$. If $\Omega_\bullet
 \cap K_B = \emptyset$, then $\partial \Omega_\bullet \subseteq B$, and,
 consequently, $\partial \Omega_\bullet = D$. If $\Omega_\bullet \cap K_B $ is
 not empty, then $\Omega_\bullet \cup K_B$ is open, connected, contains
 $\Omega$, avoids $D$ and is contained in $B$. Hence, $\Omega_\bullet \cup K_B
 \subseteq \Omega_\bullet$, what implies $\partial B \subseteq
 \partial \Omega_\bullet$.
\end{proof}

\begin{remark} \label{r-pointiii}
 At first glance one might think that $\Omega_\bullet$ could always be taken as
 $B \setminus D$. The point is that this set need not be connected, as the
 following example shows: take $\Omega = \{ \mathrm x : 1 < |\mathrm x| < 2 \}$
 and $D = \{ \mathrm x : |\mathrm x| = 1 \}$. Obviously, if a ball $B$ contains
 $\overline \Omega$, then $B \setminus D$ cannot be connected. In the spirit of
 Lemma \ref{l-1eck}, here the set $\Omega_\bullet$ has to be taken as $B
 \setminus (D \cup \{\mathrm x: |\mathrm x| < 1 \})$.
\end{remark}

The next lemma links the Hausdorff content, appearing in
Proposition~\ref{p-juha}, to the Hausdorff measure, compare also \cite{carls}.
\begin{lemma} \label{l-Lehrbaeckfix}
 If $F \subseteq \R^d$ is bounded and satisfies the Ahlford-David condition
 \eqref{e-ahlf}, then there is a $C \ge 0$ with $\mathcal H^\infty_{d-1}(E) \ge
 C \mathcal H_{d-1}(E)$ for every non-empty Borel set $E \subseteq F$.
\end{lemma}
\begin{proof}
 Let $\{B(x_j, r_j)\}_{j \in \N}$ be a covering of $E$ by open balls centered
 in $E$. If $r_j \le 1$, then $r_j^{d-1}$ is comparable to
 $\mathcal H_{d-1}(F \cap B(x_j, r_j))$, whereas if $r_j > 1$ then certainly
 $\mathcal H_{d-1}(F \cap B(x_j, r_j)) \le \mathcal H_{d-1}(F) r_j^{d-1}$. Note
 carefully that $0 < \mathcal H_{d-1}(F) < \infty$ holds, since $F$ can be
 covered by finitely many balls with radius one centered in $F$. Altogether,
 \begin{align*}
  \sum_{j= 1}^\infty r_j^{d-1} \ge C \sum_{j=1}^\infty
	\mathcal H_{d-1}(F \cap B(x_j, r_j)) \ge C \mathcal H_{d-1} \Bigl( F
	\cap \bigcup_{j=1}^\infty B(x_j,r_j) \Bigr) \ge C \mathcal H_{d-1}(E)
 \end{align*}
 with $C$ depending only on $F$. Taking the infimum,
 $\mathcal H_{d-1}^\infty(E) \ge C \mathcal H_{d-1}(E)$ follows.
\end{proof}

 Let us now prove Theorem~\ref{t-hardy}. One first observes that in both cases
 appearing in Lemma~\ref{l-1eck} the set $\partial \Omega_\bullet$ satisfies
 the Ahlfors-David condition: for the boundary part $D$ this was supposed in 
 Assumption~\ref{assu-general}, and for $\partial B$ this is obvious. Thus,
 from the Ahlfors-David condition for $\Omega_\bullet$ we get constants
 $r_\bullet > 0$ and $c > 0$ with
 \[ \mathcal H_{d-1} \bigl( \partial \Omega_\bullet \cap B(\mathrm y,r) \bigr)
	\ge c  r^{d-1}, \quad \mathrm y \in \partial \Omega_\bullet,\ r \in
	{]0, r_\bullet]}.
 \]
 This yields, invoking Lemma~\ref{l-Lehrbaeckfix},
 \begin{align} \label{e-fatn00}
  \mathcal H_{d-1}^\infty \bigl( \partial \Omega_\bullet \cap B(\mathrm y,r)
	\bigr) &\ge C \mathcal H_{d-1} \bigl( \partial \Omega_\bullet \cap
	B(\mathrm y,r) \bigr) \nonumber \\
  & \ge C c \Bigl( \frac{r_\bullet}{\mathrm{diam}(\Omega_\bullet)} \Bigr)^{d-1}
	r^{d-1}, \quad \mathrm y \in \partial \Omega_\bullet,\ r \in
	{]0, \mathrm{diam}(\Omega_\bullet)}].
 \end{align}
 But \eqref{e-fatn00} implies the inner boundary density
 condition~\eqref{e-fatn}, compare \cite[p.~2195]{juha}. Thus Proposition~\ref{p-lewis}
 and Proposition~\ref{p-juha} imply that Hardy's inequality in \eqref{e-harrd}
 is true for $\Xi = \Omega_\bullet$ and all $g \in W^{1,p}_0(\Omega_\bullet)$.

 In view of Lemma~\ref{l-1eck} we can define an extension operator
 $\mathfrak E_\bullet : W^{1,p}_D(\Omega) \to W^{1,p}_0(\Omega_\bullet)$ as
 follows: If $\partial \Omega_\bullet = D$, then we put $\mathfrak E_\bullet
 \psi := \mathfrak E \psi|_{\Omega_\bullet}$, where $\mathfrak E$ is the
 extension operator from Lemma~\ref{l-extend}. If $\partial \Omega_\bullet = D
 \cup \partial B$, then we choose an $\eta \in C_0^\infty(B)$ with $\eta \equiv
 1$ on $\overline \Omega$ and put $\mathfrak E_\bullet \psi := (\eta
 \mathfrak E \psi )|_{\Omega_\bullet}$. Now, let $f \in  W^{1,p}_D(\Omega)$.
 Then we can use \eqref{e-harrd} for $\mathfrak E_\bullet f \in
 W^{1,p}_0(\Omega_\bullet)$ and we finally find
\begin{align*}
  \int_\Omega \left| \frac{f}{\d_D} \right|^p \; \dd \mathrm{x} &\le
	\int_\Omega \left| \frac{f}{\d_{\partial \Omega_\bullet}} \right|^p
	\; \dd \mathrm{x} \le \int_{\Omega_\bullet} \left|
	\frac{\mathfrak E_\bullet f}{\d_{\partial \Omega_\bullet}} \right|^p
	\; \dd \mathrm{x} \le  c \int_{\Omega_\bullet} |
	\nabla (\mathfrak E_\bullet f)|^p \; \dd \mathrm{x} \nonumber \\
  &\le c \|f\|^p_{W^{1,p}_D} \le c \int_{\Omega} |\nabla f|^p
	\; \dd \mathrm{x}.
\end{align*}
This proves Theorem~\ref{t-hardy}.

\begin{remark} \label{r-otherproof}
 There is another strategy of proof for Hardy's inequality~\eqref{e-harrd},
 avoiding the concept of 'uniformly $p$-fat'. In \cite{juha} it is proved that
 the inner boundary density condition~\eqref{e-fatn} implies the so-called
 $p$-pointwise Hardy inequality which implies Hardy's inequality, compare also
 \cite{korte}.
\end{remark}
%
%
%
%
\section{An adapted Calder\'on-Zygmund decomposition} \label{sec-CZD}
%
%
%
%
%
%
%
%
\noindent
The proof of Theorem~\ref{t-mainsect} heavily relies on a Calder\'on-Zygmund
decomposition for $W^{1,p}_D$ functions. The important point, which brings the
mixed boundary conditions into play, is that we have to make sure that for $f
\in \mathrm{dom}_{L^p}(A_0^{1/2})$, the good and the bad part of the
decomposition are both also in this space. This is not guaranteed neither by
the classical Calder\'on-Zygmund decomposition nor by the version for Sobolev
functions in \cite[Lemma~4.12]{auschmem}. This problem will be solved by
incorporating the Hardy inequality into the decomposition.

For ease of notation, in the whole section we set $1/\d_\emptyset = 0$.

We denote by $\QQ$ the set of all closed axis-parallel cubes, i.e.\@ all sets
of the form $\{ \mathrm x \in \R^d : |\mathrm x - \mathrm m|_\infty \le \ell/2
\}$ for some midpoint $\mathrm m \in \R^d$ and sidelength $\ell >
0$. In the following, for a given cube $Q \in \QQ$ we will often write $s Q$
for some $s > 0$, meaning the cube with the same midpoint $\mathrm m$, but
sidelength $s \ell$ instead of $\ell$.

Furthermore, for every $\mathrm x \in \R^d$ we set $\QQ_{\mathrm x} :=
 \{ Q \in \QQ : \mathrm x \in Q^\circ \}$. Now we may define the
Hardy-Littlewood maximal operator $M$ for all $\varphi \in L^1(\R^d)$ by
\begin{equation} \label{e-HLMO}
 (M \varphi)(\mathrm x) = \sup_{Q \in \QQ_{\mathrm x}} \frac{1}{|Q|} \int_Q
	|\varphi|, \qquad \mathrm x \in \R^d.
\end{equation}
It is well known (see \cite[Ch.~1]{Ste70}) that $M$ is of weak type $(1,1)$,
so there is some $K > 0$, such that for all $p \ge 1$
\begin{equation} \label{e-Maxfunk}
  \bigl| \bigl\{ \mathrm x \in \R^d : |[M (|\varphi|^p)](\mathrm x)| > \alpha^p
	\bigr\} \bigr| \le \frac{K}{\alpha^p} \|\varphi\|_{L^p(\R^d)}^p, \quad
	\text{for all } \alpha > 0 \text{ and } \varphi \in L^p(\R^d).
\end{equation}
%
%
\begin{lemma} \label{CZZ}
Let $\Omega$ and $D$ satisfy Assumption~\ref{assu-general}. Let $p \in {]1,\infty[}$,
$f \in W^{1,p}_D$ and $\alpha > 0$ be given. Then there exist an at most countable index
set $I$, cubes $Q_j \in \QQ$, $j \in I$, and measurable functions $g, b_j : \Omega 
\to \R$, $j \in I$, such that for some constant $N \ge 0$, independent of $\alpha$ and $f$,

\renewcommand{\theenumi}{(\arabic{enumi})}
\renewcommand{\labelenumi}{\theenumi}

\begin{enumerate}
\item $\displaystyle f = g + \sum_{j \in I} b_j$,\label{CZZ(1)}
\item $\| \nabla g \|_{L^\infty} + \|g\|_{L^\infty} + \| g/ \d_D \|_{L^\infty}
	\le N \alpha$,\label{CZZ(2)}
\item $\displaystyle \supp(b_j) \subseteq Q_j, \ b_j \in W^{1,1}_D \cap W^{1,p}
        \text{ and } \int_\Omega \Bigl( |\nabla b_j| + |b_j| +
	\frac{|b_j|}{\d_D} \Bigr) \le N \alpha |Q_j|$ for every $j \in
	I$,\label{CZZ(3)}
\item $\displaystyle \sum_{j \in I} |Q_j| \le \frac{N}{\alpha^p}
	\|f\|_{W^{1,p}_D}^p$,\label{CZZ(4)}
\item $\displaystyle \sum_{j \in I} {\bf 1}_{Q_j}(\mathrm x) \le N$ for all
	$\mathrm x \in \R^d$,\label{CZZ(5)}
\item $\|g\|_{W^{1,p}_D} \le N \|f \|_{W^{1,p}_D}$.\label{CZZ(6)}
\end{enumerate}
If $D \neq \emptyset$, all the norms $\| f \|_{W^{1,p}_D}$ may be replaced by
$\| \nabla f\|_{L^p}$.
\end{lemma}
%
\renewcommand{\theenumi}{(\roman{enumi})}
\renewcommand{\labelenumi}{\theenumi}

%
%
In order to verify the final statement, note that for $D \neq \emptyset$ the
Ahlfors-David condition guarantees that the surface measure of $D$ is strictly
positive, cf.\@ Remark~\ref{r-surfmeas}~\ref{r-surfmeas:ii}. Thus we can
conclude by Remark~\ref{r-equivnorms}.

We will subdivide the proof of Lemma~\ref{CZZ} into six steps.

\subsection*{Step 1: Adapted Maximal function}

Let $f \in W^{1,p}_D$. Then, using the extension operator $\mathfrak E_\bullet$
from the proof of Theorem~\ref{t-hardy}, we find $\mathfrak E_\bullet f \in
W^{1,p}_0(\Omega_\bullet)$. So we may extend this function again by zero to
the whole of $\R^d$, obtaining a function $\tilde f \in W^{1,p}_D(\R^d)$ that
satisfies $\mathrm{supp}(\tilde f) \subseteq B$ for the ball $B$ from
Section~\ref{sec-Hardy} and the estimate $\|\tilde f \|_{W^{1,p}_D(\R^d)} \le C
\|f\|_{W^{1,p}_D}$ with a constant $C$ that does not depend on $f$.
Furthermore, Hardy's inequality
\begin{equation}\label{e-TildeHardy}
  \| \tilde f/\d_D \|_{L^p(\R^d)} \le C \|\nabla \tilde f\|_{L^p(\R^d)}
\end{equation}
holds, cf.\@ Section~\ref{sec-Hardy}.

\begin{remark}
 Using $\tilde f$, we will construct the Calder\'on-Zygmund decomposition on
 all of $\R^d$ and afterwards restrict again to $\Omega$. Admittedly, it would
 be more natural to stay inside $\Omega$, but this leads to several technical
 problems, since the regularity of the boundary of cubes in $\Omega$,
 i.e.\@ $\Omega \cap Q$ for some cube $Q$ in $\R^d$, may be very low, so that
 for instance the validity of the Poincar\'e inequality is no longer obvious.
 If $\Omega$ is more regular, say a strong Lipschitz domain, this extension
 can be omitted.
\end{remark}

We consider the open set
\[ E := \bigl\{ \mathrm x \in \R^d : \bigl[ M \bigl( |\nabla \tilde f| +
	|\tilde f| + |\tilde f|/\d_D \bigr) \bigr] (\mathrm x) > \alpha
	\bigr\}.
\]

The easiest case is that of $E = \emptyset$. Then we may take $I = \emptyset$
and $g = f$ and the only assertion we have to show is \ref{CZZ(2)}, the rest
being trivial. So, let $\mathrm x \in \Omega$ be given. Since $\mathrm x$ is
not in $E$, we have for almost all such $\mathrm x$, by the fact that
$h(\mathrm x)  \le (Mh)(\mathrm x)$ for all Lebesgue points of an $L^1(\R^d)$
function $h$,
\begin{align*}
  |\nabla g(\mathrm x)| + |g(\mathrm x)| + |g(\mathrm x)|/\d_D(\mathrm x) &=
	|\nabla f(\mathrm x)| + |f(\mathrm x)| + |f(\mathrm x)|/\d_D(\mathrm x)
	\\
  &= |\nabla \tilde f(\mathrm x)| + |\tilde f(\mathrm x)| +
	|\tilde f(\mathrm x)|/\d_D(\mathrm x) \\
  &\le \bigl[ M \bigl( |\nabla \tilde f| + |\tilde f| + |\tilde f|/\d_D \bigr)
	\bigr] (\mathrm x) \le \alpha.
\end{align*}
This implies \ref{CZZ(2)}.

\medskip

So, we turn to the case $E \neq \emptyset$. By Jensen's inequality,
\eqref{e-Maxfunk}, \eqref{e-TildeHardy} and the continuity of the extension
operator we obtain
\begin{align}
  |E| &\le \bigl| \bigl\{ \mathrm x \in \R^d : \bigl[ M \bigl( (
	|\nabla \tilde f| + |\tilde f| + |\tilde f|/\d_D )^p \bigr) \bigr]
	(\mathrm x) > \alpha^p \bigr\} \bigr| \nonumber \\
  &\le \frac{K}{\alpha^p} \bigl\| |\nabla \tilde f| + |\tilde f| +
	|\tilde f|/\d_D \bigr\|_{L^p(\R^d)}^p \le \frac{C}{\alpha^p} \|\tilde f\|_{W^{1,p}(\R^d)}^p \le \frac{C}{\alpha^p}
	\|f\|_{W^{1,p}_D}^p.\label{e-|E|}
\end{align}

In particular this measure is finite, so $F : = \R^d \setminus E \neq
\emptyset$. This allows for choosing a Whitney decomposition of $E$,
cf.\@ \cite[Lemmas~5.5.1 and 5.5.2]{BS88}, see also \cite{Ste70} and
\cite{Tor86}. Thus, we get an at most countable index set $I$ and a collection
of cubes $Q_j \in \QQ$, $j \in I$, with sidelength $\ell_j$ that fulfill the
following properties for some $c_1, c_2 \ge 1$

\medskip

\begin{minipage}[b]{.48\textwidth}
\begin{enumerate}
 \item $\displaystyle E = \bigcup_{j \in I} \textstyle \frac 89 Q_j$.
 \item $\displaystyle \frac 89 Q_j^\circ \cap \frac 89 Q_k^\circ =
	\emptyset$ for all $j,k \in I$, $j \neq k$.
 \item $\displaystyle Q_j \subseteq E$ for all $j \in I$.
\end{enumerate}
\end{minipage}
\begin{minipage}[b]{.48\textwidth}
\begin{enumerate}
 \setcounter{enumi}{3}%
 \item \label{pr-Whitney:iv} $\displaystyle \sum_{j \in I} {\mathbf 1}_{Q_j}
	\le c_1$.
 \item \label{pr-Whitney:v} $\displaystyle \frac{1}{c_2} \ell_j \le \d(Q_j, F)
	\le c_2 \ell_j$ for all $j \in I$.
\end{enumerate}
\end{minipage}

\medskip

There are two immediate consequences of these properties that are important to
observe. Firstly, the family $Q_j^\circ$, $j \in I$, is an open covering of $E$
and, secondly, \ref{pr-Whitney:v} implies that for some $\tilde c > 1$,
independent of $j$, we have
\begin{equation}\label{e-QjSchnittF}
  (\tilde c Q_j) \cap F \neq \emptyset \quad \text{for all } j  \in I.
\end{equation}

Now, \ref{pr-Whitney:iv} immediately implies \ref{CZZ(5)} and this, together
with \eqref{e-|E|} allows to prove \ref{CZZ(4)} due to
\[ \sum_{j \in I} |Q_j| = \int_E \sum_{j \in I} {\mathbf 1}_{Q_j} \le c_1 |E|
	\le \frac{C}{\alpha^p} \|f\|_{W^{1,p}_D}^p.
\]

\subsection*{Step 2: Definition of the good and bad functions}
Let $(\varphi_j)_{j \in I}$ be a partition of unity on $E$ with

\medskip

\begin{minipage}[t]{.48\textwidth}
\begin{itemize}
\item[a)] $\varphi_j \in C^\infty(\R^d)$,
\item[b)] $\supp(\varphi_j) \subseteq Q_j^\circ$,
\end{itemize}
\end{minipage}
\begin{minipage}[t]{.48\textwidth}
\begin{itemize}
\item[c)] $\varphi_j \equiv 1$ on $\frac 89 Q_j$,
\item[d)] $\|\varphi_j\|_{L^\infty} + \ell_j \|\nabla \varphi_j \|_{L^\infty}
	\le c$,
\end{itemize}
\end{minipage}

\medskip

\noindent
for all $j \in I$ and some $c > 0$. The construction of such a partition can
be found e.g.\@ in \cite[Section~5.5]{BS88}.

Let us distinguish two types of cubes $Q_j$. We say that $Q_j$ is a
\emph{usual} cube, if $\d(Q_j, D) \ge \ell_j$ and $Q_j$ is a \emph{special}
cube, if $\d(Q_j, D) < \ell_j$ (In the case $D = \emptyset$ all cubes are
seen as usual ones). Then we define for every $j \in I$, using the
notation $h_Q := \frac{1}{|Q|} \int_Q h$,
\[ \tilde b_j := \begin{cases}
           \bigl( \tilde f - \tilde f_{Q_j} \bigr) \varphi_j, &
		\text{if $Q_j$ is usual},\\
	   \tilde f \varphi_j, & \text{if $Q_j$ is special}.
          \end{cases}
\]
Setting $\tilde g : = \tilde f - \sum_{j \in I} \tilde b_j$ as well as $b_j :=
\tilde b_j|_\Omega$ and $g := \tilde g|_\Omega$, these functions automatically
satisfy \ref{CZZ(1)}. Note that there is no problem of convergence in this
sum, due to \ref{CZZ(5)}.

It is clear by construction that $\supp(b_j) \subseteq Q_j$ and $b_j \in
W^{1,p}(\Omega)$ for all $j \in I$. The next step is to show that $b_j \in
W^{1,1}_D$ and since $W^{1,p} \hookrightarrow W^{1,1}$, we only have to
establish the right boundary behaviour of $b_j$.

We start with the case of a usual cube $Q_j$. Then $b_j = \bigl( (\tilde f -
\tilde f_{Q_j} ) \varphi_j \bigr) |_\Omega$. Since $\varphi_j$ has support in
$Q_j$ and $\dd(Q_{j}, D) \ge \ell_j > 0$, the function $b_j$ can be
approximated by $C^{\infty}_c(\R^d \setminus D)$ functions in the norm of
$W^{1,1}$. Thus $b_j \in W^{1,1}_D$.

If $Q_j$ is  a special cube, we have $b_j = (\tilde f \varphi_j)|_\Omega$. The
fact that $\tilde f \in W^{1,p}_D(\R^d)$ implies that there is a sequence
$(\tilde f_k)_k \subseteq C^\infty_c(\R^d \setminus D)$, such that $\tilde f_k
\to \tilde f$ in $W^{1,p}(\R^d)$. Therefore, $(\tilde f_k \varphi_j)_k$ is a
sequence in $C^\infty_c(\R^d \setminus D)$ and we show that it converges to
$\tilde f \varphi_j$ in $W^{1,1}$, so that we can conclude that $b_j \in W^{1,1}_D$.
This convergence follows from $\varphi_j \in W^{1,p'}(\R^d)$ by
\[ \|\tilde f \varphi_j - \tilde f_k \varphi_j\|_{L^1} \le \|\tilde f -
	\tilde f_k \|_{L^p} \|\varphi_j\|_{L^{p'}} \to 0 \quad (k \to \infty)
\]
and the corresponding estimate for the gradient
\begin{align*}
  \bigl\| \nabla( \tilde f \varphi_j) - \nabla (\tilde f_k \varphi_j)
	\bigr\|_{L^1} &\le \bigl\| \nabla (\tilde f - \tilde f_k) \varphi_j
	\bigr\|_{L^1} + \bigl\| (\tilde f - \tilde f_k) \nabla \varphi_j
	\bigr\|_{L^1} \\
  &\le \bigl\| \nabla (\tilde f - \tilde f_k) \bigr\|_{L^p}
	\|\varphi_j\|_{L^{p'}} + \|\tilde f - \tilde f_k\|_{L^p}
	\|\nabla \varphi_j \|_{L^{p'}} \to 0 \quad (k \to \infty).
\end{align*}

\subsection*{Step 3: Proof of \ref{CZZ(3)}}

After the above considerations, it remains to prove the estimate. We start
again with the case of a usual cube and for later purposes we introduce some
$q \in [1, \infty[$. On usual cubes it holds $\nabla \tilde b_j = \nabla
\tilde f \varphi_j + ( \tilde f - \tilde f_{Q_j} ) \nabla \varphi_j$ and
using d) we obtain
\begin{align*}
  \int_{Q_j} |\nabla \tilde b_j|^q &\le \int_{Q_j} \bigl( |\nabla \tilde f|
	|\varphi_j| + |\tilde f - \tilde f_{Q_j}| |\nabla \varphi_j| \bigr)^q
	\le C \int_{Q_j} \bigl( |\nabla \tilde f|^q |\varphi_j|^q + |\tilde f
	- \tilde f_{Q_j}|^q |\nabla \varphi_j|^q \bigr) \nonumber \\
  &\le C \Bigl( \int_{Q_j} |\nabla \tilde f|^q + \frac{1}{\ell_j^q} \int_{Q_j}
	|\tilde f - \tilde f_{Q_j}|^q \Bigr).
\end{align*}
In the second integral we may now apply the Poincar\'e inequality, since
$\tilde f - \tilde f_{Q_j}$ has zero mean on $Q_j$. This yields
\begin{equation} \label{e-nablabj}
 \int_{Q_j} |\nabla \tilde b_j|^q \le C \Bigl( \int_{Q_j} |\nabla \tilde f|^q +
	\frac{1}{\ell_j^q} \mathrm{diam}(Q_j)^q \int_{Q_j} |\nabla \tilde f|^q
	\Bigr) \le C \int_{Q_j} |\nabla \tilde f|^q.
\end{equation}
We now specialize again to $q = 1$ and, invoking \eqref{e-QjSchnittF}, we pick
some $\mathrm z \in \tilde c Q_j\cap F$, and bring into play the maximal
operator:
\begin{align}
  \int_{Q_j} |\nabla \tilde b_j| &\le C \int_{\tilde c Q_j} |\nabla \tilde f
	| \le C |Q_j| \frac{1}{|\tilde c Q_j|} \int_{\tilde c Q_j} \Bigl(
	|\nabla \tilde f| + |\tilde f| + \frac{|\tilde f|}{\d_D} \Bigr)
	\label{e-(3).1} \\
&\le C |Q_j| \sup_{Q \in \QQ_z} \frac{1}{|Q|} \int_{Q} \Bigl(
	|\nabla \tilde f| + |\tilde f| + \frac{|\tilde f|}{\d_D} \Bigr) = C
	|Q_j| \biggl[ M \Bigl( |\nabla \tilde f| + |\tilde f| +
	\frac{|\tilde f|}{\d_D} \Bigr) \biggr](\mathrm z). \nonumber
\end{align}
Now, we capitalize that $\mathrm z \in F$ and obtain
\begin{equation}\label{e-nablabj1}
  \int_{\Omega} |\nabla b_j| \le \int_{Q_j} |\nabla \tilde  b_j| \le C |Q_j|
	\alpha.
\end{equation}
For the corresponding estimate for $|b_j|$ we use again the Poincar\'e
inequality for $\tilde f - \tilde f_{Q_j}$ on $Q_j$ to obtain for all $q \in
{[1, \infty[}$
\begin{equation} \label{e-b_jwithq}
  \int_\Omega |b_j|^q \le \int_{Q_j} |\tilde b_j|^q = \int_{Q_j} | \tilde f -
	\tilde f_{Q_j} |^q |\varphi_j|^q \le C \int_{Q_j} |\tilde f -
	\tilde f_{Q_j}|^q \le C \int_{Q_j} |\nabla \tilde f|^q.
\end{equation}
Note that the factor $\mathrm{diam}(Q_j)$ from the Poincar\'e inequality is
bounded uniformly in $j$, since all $Q_j$ are contained in $E$ and $E$ has
finite measure.

Proceeding as in \eqref{e-(3).1} and \eqref{e-nablabj1}, we find, specialising
to $q = 1$,
\begin{equation} \label{e-b_j}
\int_\Omega |b_j| \le C |Q_j| \alpha.
\end{equation}

For the third term $|b_j|/\d_D$ we note that on a usual cube $Q_j$ we have
$\d_D \ge \ell_j$. Thus we get as before by the Poincar\'e inequality
\[ \int_\Omega \frac{|b_j|}{\d_D} \le \int_{Q_j} \frac{|\tilde b_j|}{\d_D} \le
	\frac{C}{\ell_j} \int_{Q_j} |\tilde f - \tilde f_{Q_j}| \le C
	\int_{Q_j} |\nabla \tilde f|
\]
and we can again conclude as in \eqref{e-(3).1} and \eqref{e-nablabj1}.

So, we turn to the proof of the estimate in \ref{CZZ(3)} for the case of a
special cube. Then $b_j = (\tilde f \varphi_j)|_\Omega$, and we get with the
help of d)
\[ |\nabla \tilde b_j| \le |\nabla \tilde f||\varphi_j| + |\tilde f |
	|\nabla \varphi_j| \le C \Bigl( |\nabla \tilde f| +
	\frac{|\tilde f|}{\ell_j} \Bigr).
\]
Since $Q_j$ is a special cube, we get for every $\mathrm x \in Q_j$
\begin{equation}\label{e-specialcube}
  \d_D(\mathrm x) = \d(\mathrm x, D) \le \mathrm{diam}(Q_j) + \d(Q_j, D) \le C
	\ell_j + \ell_j \le C \ell_j
\end{equation}
and this in turn yields
\begin{equation} \label{e-nablabjspecial}
  |\nabla \tilde b_j| \le C \Bigl( |\nabla \tilde f| + \frac{|\tilde f|}{d_D}
	\Bigr). 
\end{equation}
Since, obviously
\begin{equation} \label{e-b_jspecial}
  |\tilde b_j| = |\tilde f \varphi_j| \le C |\tilde f| \qquad \text{and}
	\qquad \frac{|\tilde b_j|}{d_D} = \frac{|\tilde f \varphi_j|}{\d_D}
	\le C \frac{|\tilde f|}{\d_D} 
\end{equation}
hold, we find by one more repetition of the arguments in \eqref{e-(3).1} and
\eqref{e-nablabj1} with some $\mathrm z \in \tilde c Q_j \cap F$
\begin{align}
 \int_\Omega \Bigl( |b_j| + |\nabla b_j| + \frac{|b_j|}{\d_D} \Bigr) &\le
	C \int_{Q_j} \Bigl( |\tilde f| + |\nabla \tilde f| +
	\frac{|\tilde f|}{\d_D} \Bigr) \nonumber \\
  &\le \frac{C|Q_j|}{|\tilde cQ_j|} \int_{\tilde c Q_j} \Bigl( |\tilde f| +
	|\nabla \tilde f| + \frac{|\tilde f|}{\d_D} \Bigr) \le C |Q_j| \alpha.
	\label{e-bjsp2}
\end{align}

\subsection*{Step 4: Proof of \ref{CZZ(2)}: Estimate of $|g|$ and $|g|/\d_D$}

The asserted bound for $|g|$ and $|g|/\d_D$ is rather easy to obtain on $F
\cap \Omega$, since on $F$ all functions $\tilde b_j$, $j \in I$, vanish, which
means $\tilde g = \tilde f$ on $F$. This implies for almost all $\mathrm x \in
F \cap \Omega$ by the definition of $F$
\[ |g(\mathrm x)| + \frac{|g(\mathrm x)|}{\d_D(\mathrm x)} =
	|\tilde f(\mathrm x)| + \frac{|\tilde f(\mathrm x)|}{\d_D(\mathrm x)}
	\le \biggl[ M \Bigl( |\nabla \tilde f| + |\tilde f| +
	\frac{|\tilde f|}{\d_D} \Bigr) \biggr](\mathrm x) \le
	\alpha.
\]
So, for the estimate of these two terms we concentrate on the case $\mathrm x
\in E$.
Setting $I_u := \{ j \in I : Q_j \text{ usual} \}$ and $I_s := \{ j \in I :
Q_j \text{ special} \}$, we obtain on $E$
\begin{align*}
  \tilde g &= \tilde f - \sum_{j \in I_u} \tilde b_j - \sum_{j \in I_s}
	\tilde b_j = \tilde f - \sum_{j \in I_u} ( \tilde f - \tilde f_{Q_j} )
	\varphi_j - \sum_{j \in I_s} \tilde f \varphi_j = \tilde f - \tilde f
	\sum_{j \in I} \varphi_j + \sum_{j \in I_u} \tilde f_{Q_j} \varphi_j \\
  &= \tilde f {\mathbf 1}_F + \sum_{j \in I_u} \tilde f_{Q_j} \varphi_j =
	\sum_{j \in I_u} \tilde f_{Q_j} \varphi_j.
\end{align*}
Now, we fix some $\mathrm x \in E$. Let $I(\mathrm x) := \{ j \in I : \mathrm x
\in \supp(\varphi_j) \}$, $I_{u,\mathrm x} := I_u \cap I(\mathrm x)$ and
$I_{s,\mathrm x} := I_s \cap I(\mathrm x)$. Then the above estimate yields
together with d)
\begin{align}
  |\tilde g(\mathrm x)| &\le \sum_{j \in I_u} |\tilde f_{Q_j}|
	|\varphi_j(\mathrm x)| \le C \sum_{j \in I_{u, \mathrm x}}
	|\tilde f_{Q_j}| = C \sum_{j \in I_{u, \mathrm x}} \frac{1}{|Q_j|}
	\Bigl| \int_{Q_j} \tilde f(\mathrm y) \; \dd \mathrm y \Bigr| \nonumber
	\\
  &\le C \sum_{j \in I_{u, \mathrm x}} \frac{1}{|Q_j|} \int_{Q_j}
	|\tilde f(\mathrm y)| \; \dd \mathrm y.
	\label{e-|g|}
\end{align}
Picking again some $\mathrm z_j \in \tilde c Q_j \cap F$, $j \in I$, this
yields with the argument that we used already several times and since
$I_{u, \mathrm x}$ is finite
\[ 
|\tilde g(\mathrm x)| \le C \sum_{j \in I_{u, \mathrm x}}
	\frac{1}{|\tilde c Q_j|} \int_{\tilde c Q_j} |\tilde f(\mathrm y)|
	\; \dd \mathrm y \le C \sum_{j \in I_{u, \mathrm x}} \bigl[ M(
	|\tilde f|) \bigr](\mathrm z_j) \le C \sum_{j \in I_{u, \mathrm x}}
	\alpha \le C \alpha.
\]
In order to estimate $\tilde g/\d_D$ on $E$, we estimate as in \eqref{e-|g|}
for $\mathrm x \in E$
\[ \frac{|\tilde g(\mathrm x)|}{\d_D(\mathrm x)} =
	\frac{\bigl| \sum_{j \in I_u} \tilde f_{Q_j} \varphi_j(\mathrm x)
	\bigl|}
	{\d_D(\mathrm x)} \le C \sum_{j \in I_{u, \mathrm x}}
	\frac{|\tilde f_{Q_j}|}{\d_D(\mathrm x)} \le C
	\sum_{j \in I_{u, \mathrm x}} \frac{1}{|Q_j|} \int_{Q_j}
	\frac{|\tilde f(\mathrm y)|}{\d_D(\mathrm x)} \; \dd \mathrm y.
\]
Every cube in this sum is a usual one, so $\d(Q_j, D) \ge \ell_j$.
Furthermore, we have $\mathrm x \in Q_j$ for all $j \in I_{u, \mathrm x}$ by
construction. This means that for every $j \in I_{u, \mathrm x}$ and all
$\mathrm y \in Q_j$ the distance between $\mathrm x$ and $\mathrm y$ is less
than $C \ell_j$ for some constant $C$ depending only on the dimension. Thus
\[ \d_D(\mathrm y) = \d(\mathrm y, D) \le \d(\mathrm y, \mathrm x) +
	\d(\mathrm x, D) \le C \ell_j + \d_D(\mathrm x) \le C \d(Q_j, D) +
	\d_D(\mathrm x) \le C \d_D(\mathrm x).
\]
Consequently, we get for some $\mathrm z_j \in \tilde c Q_j \cap F$ as before
\begin{align*}
  \frac{|g(\mathrm x)|}{\d_D(\mathrm x)} &\le C \sum_{j \in I_{u, \mathrm x}}
	\frac{1}{|Q_j|} \int_{Q_j}
	\frac{|\tilde f(\mathrm y)|}{\d_D(\mathrm y)} \; \dd \mathrm y \le
	C \sum_{j \in I_{u,\mathrm x}} \frac{1}{|\tilde c Q_j|}
	\int_{\tilde c Q_j} \frac{|\tilde f(\mathrm y)|}{\d_D(\mathrm y)}
	\; \dd \mathrm y \\
  &\le C \sum_{j \in I_{u,\mathrm x}} \bigl[ M(|\tilde f|/\d_D) \bigr]
	(\mathrm z_j) \le C \alpha.
\end{align*}

\subsection*{Step 5: Proof of \ref{CZZ(2)}: Estimate of $|\nabla g|$}

In order to estimate $|\nabla g|$, it is not sufficient to know that
$\sum_{j \in I} \tilde b_j$ converges pointwise as before. At least we have to
know some convergence in the sense of distributions to push the gradient
through the sum. Let $J \subseteq I$ be finite. Then we have, due to
\eqref{e-b_j} for usual cubes and \eqref{e-bjsp2} for special cubes
\[ \Bigl\| \sum_{j \in J} |\tilde b_j| \Bigr\|_{L^1(\R^d)} = \int_{\R^d}
	\sum_{j \in J} |\tilde b_j| = \sum_{j \in J} \int_{Q_j} |\tilde b_j|
	\le C \alpha \sum_{j \in J} |Q_j|
\]
with a constant $C$ that is independent of the choice of $J$. Since
$\sum_{j \in I} |Q_j|$ is convergent due to \ref{CZZ(4)}, this implies that
$\sum_{j \in I} |\tilde b_j|$ is a Cauchy sequence in $L^1(\R^d)$.

In particular $\sum_{j \in I} \tilde b_j$ converges in the sense of
distributions, so we get $\nabla \sum_{j\in I} \tilde b_j = \sum_{j \in I}
\nabla \tilde b_j$ in the sense of distributions.

In a next step we show that the sum $\sum_{j \in I} \nabla \tilde b_j$
converges absolutely in $L^1$. Investing the estimates in \eqref{e-nablabj}
and \eqref{e-nablabjspecial}, respectively, we find

\[ \int_{Q_j} |\nabla \tilde b_j| \le C \int_{Q_j} \Bigl( |\nabla \tilde f| +
	\frac{|\tilde f|}{\d_D} \Bigr).
\]
Thus, we obtain by \ref{CZZ(5)} and the fact that $E$ has finite measure,
cf.\@ \eqref{e-|E|},
\begin{align*}
  \sum_{j \in I} \| \nabla \tilde b_j \|_{L^1(\R^d)} &= \sum_{j \in I} \|
	\nabla \tilde b_j \|_{L^1(Q_j)} \le C \sum_{j \in I} \int_{Q_j} \Bigl(
	|\nabla \tilde f| + \frac{|\tilde f|}{\d_D} \Bigr) = C \int_E
	\sum_{j \in I} \mathbf{1}_{Q_j} \Bigl( |\nabla \tilde f| +
	\frac{|\tilde f|}{\d_D} \Bigr) \\
  &\le C \Bigl\| |\nabla \tilde f| + \frac{|\tilde f|}{\d_D}
	\Bigr\|_{L^1(E)} \le C \Bigl\| |\nabla \tilde f| +
	\frac{|\tilde f|}{\d_D} \Bigr\|_{L^p(E)} \le \bigl\| \nabla \tilde f
	\bigr\|_{L^p(\R^d)}+ \Bigl\| \frac{\tilde f}{\d_D} \Bigr\|_{L^p(\R^d)}.
\end{align*}
Now, by Hardy's inequality~\eqref{e-TildeHardy} this last expression is finite
and this yields the desired absolute convergence.

This allows us to calculate
\[ \nabla \tilde g = \nabla \tilde f - \sum_{j \in I} \nabla \tilde b_j =
	\nabla \tilde f - \sum_{j \in I_u} \bigl( \nabla \tilde f \varphi_j +
	(\tilde f - \tilde f_{Q_j}) \nabla \varphi_j \bigr) - \sum_{j \in I_s}
	\bigl( \nabla \tilde f \varphi_j + \tilde f \nabla \varphi_j \bigr).
\]
Note that the above considerations concerning the convergence of
$\sum_{j \in I} \nabla \tilde b_j$ also yield that the sums over
$\nabla \tilde f \varphi_j$, $(\tilde f - \tilde f_{Q_j}) \nabla \varphi_j$ and
$\tilde f \nabla \varphi_j$ are absolutely convergent in $L^1$, so
\[ \nabla \tilde g = \nabla \tilde f - \sum_{j \in I} \nabla \tilde f
	\varphi_j - \sum_{j \in I_u} (\tilde f - \tilde f_{Q_j})
	\nabla \varphi_j - \sum_{j \in I_s} \tilde f \nabla \varphi_j =
	\nabla \tilde f \mathbf{1}_F - \sum_{j \in I_u} (\tilde f -
	\tilde f_{Q_j}) \nabla \varphi_j - \sum_{j \in I_s} \tilde f
	\nabla \varphi_j.
\]
On $F$ we know that every summand in the above two sums vanishes, so by the
$L^1$-convergence shown above we see $\nabla \tilde g = \nabla \tilde f$ on
$F$. Thus on $F$ we easily get the desired $L^\infty$-estimate for
$\nabla \tilde g$, since for almost all $\mathrm x \in F$
\[ |\nabla \tilde g(\mathrm x)| = |\nabla \tilde f(\mathrm x)| \le
	M(|\nabla \tilde f|)(\mathrm x) \le M \bigl( |\nabla \tilde f| +
	|\tilde f| + |\tilde f|/\dd_D \bigr)(\mathrm x) \le \alpha.
\]
So, we concentrate on $\mathrm x \in E$. Since $E$ is open all sums in
\[ \nabla \tilde g(\mathrm{x}) = - \sum_{j \in I_u} \bigl( \tilde f(\mathrm x)
	- \tilde f_{Q_j} \bigr) \nabla \varphi_j(\mathrm x) - \sum_{j \in I_s}
	\tilde f(\mathrm x) \nabla \varphi_j(\mathrm x)
\]
are finite thanks to \ref{CZZ(5)} and $\sum_{j \in I} \varphi_j$ is constantly
$1$ in a neighbourhood of $\mathrm x$. Thus, we may calculate for $\mathrm x
\in E$
\[ \nabla \tilde g(\mathrm x) = \sum_{j \in I_u} \tilde f_{Q_j} \nabla
	\varphi_j(\mathrm x) - \tilde f(\mathrm x) \sum_{j \in I}
	\nabla \varphi_j(\mathrm x) = \sum_{j \in I_u} \tilde f_{Q_j}
	\nabla \varphi_j(\mathrm x).
\]
We set on $E$
\[ h_u := \sum_{j \in I_u} \tilde f_{Q_j} \nabla \varphi_j \qquad \text{and}
 	\qquad h_s := \sum_{j \in I_s} \tilde f_{Q_j} \nabla \varphi_j
\]
and we will show in the following the estimates $|h_s(\mathrm x)| \le C
\alpha$ and $|h_u(\mathrm x) + h_s(\mathrm x)| \le C \alpha$ for all $\mathrm x
\in E$. Then we have the same bound for $h_u$ and hence also for
$\nabla \tilde g$ on $E$.

In order to show the desired estimate for $h_s$, we recall that by
\eqref{e-specialcube} we have $\dd_D(\mathrm y) \le C \ell_j$ for all
$\mathrm y$ in a special cube $Q_j$. Using d) and this estimate we find for
all $\mathrm x \in E$
\begin{align*}
 |h_s(\mathrm x)| &\le \sum_{j \in I_s} |\tilde f_{Q_j}|
	|\nabla \varphi_j(\mathrm x)| \le \sum_{j \in I_{s, \mathrm x}}
	\frac{C}{\ell_j}|\tilde f_{Q_j}| \le C \sum_{j \in I_{s, \mathrm x}}
	\frac{1}{|Q_j|} \int_{Q_j} \frac{|\tilde f(\mathrm y)|}{\ell_j}
	\; \d \mathrm y \\
  & \le C \sum_{j \in I_{s, \mathrm x}} \frac{1}{|Q_j|} \int_{Q_j}
	\frac{|\tilde f(\mathrm y)|}{\d_D(\mathrm y)} \; \d \mathrm y.
\end{align*}
Now, we use again that the above sum is finite, uniformly in $\mathrm x$, so
it suffices to estimate each addend by $C \alpha$. In order to do so, we once
more bring into play the maximal operator in some point $\mathrm z_j \in
\tilde c Q_j \cap F$:
\[ \frac{1}{|Q_j|} \int_{Q_j} \frac{|\tilde f(\mathrm y)|}{\d_D(\mathrm y)}
	\; \d \mathrm y \le C \frac{1}{|\tilde c Q_j|} \int_{\tilde c Q_j}
	\frac{|\tilde f(\mathrm y)|}{\d_D(\mathrm y)} \; \d \mathrm y
	\le C M \bigl( |\nabla \tilde f| + |\tilde f| + |\tilde f|/\dd_D \bigr)
	(\mathrm z_j) \le C \alpha.
\]
We turn to the estimate of $h_u + h_s$. Let $\mathrm x \in E$ and choose some
$i_0 \in I(\mathrm x)$. Then for every $j \in I(\mathrm x)$ we have $\mathrm x
\in Q_j \cap Q_{i_0}$, so by property (v) of the Whitney cubes, the sidelengths
$\ell_j$ and $\ell_{i_0}$ are comparable with uniform constants. Thus we can
choose some $\kappa \ge \tilde c$, such that $\kappa Q_{i_0} \supseteq Q_j$ for
all $j \in I(\mathrm x)$. Since $\sum_{j \in I} \nabla \varphi_j(\mathrm x) =
0$, one finds
\[ (h_u + h_s) (\mathrm x) = \sum_{j \in I} \tilde f_{Q_j}
	\nabla \varphi_j(\mathrm x) = \sum_{j \in I} (\tilde f_{Q_j} -
	\tilde f_{\kappa Q_{i_0}}) \nabla \varphi_j(\mathrm x).
\]
This implies thanks to d)
\[ \bigl| (h_u + h_s) (\mathrm x) \bigr| \le \sum_{j \in I} |\tilde f_{Q_j} -
	\tilde f_{\kappa Q_{i_0}} |
	|\nabla \varphi_j(\mathrm x)| \le \sum_{j \in I(\mathrm x)}
	\frac{C}{\ell_j} |\tilde f_{Q_j} -
	\tilde f_{\kappa Q_{i_0}} |.
\]
For every $j \in I(\mathrm x)$ we have
\begin{align*}
  |\tilde f_{Q_j} - \tilde f_{\kappa Q_{i_0}}| &= \Bigl|
	\frac{1}{|Q_j|} \int_{Q_j} \tilde f(\mathrm y) \; \d \mathrm y -
	\tilde f_{\kappa Q_{i_0}} \Bigr| = \Bigl|
	\frac{1}{|Q_j|} \int_{Q_j} \bigl( \tilde f(\mathrm y) -
	\tilde f_{\kappa Q_{i_0}} \bigr) \; \d \mathrm y
	\Bigr| \\
  &\le \frac{1}{|Q_j|} \int_{Q_j} \bigl| \tilde f(\mathrm y) -
	\tilde f_{\kappa Q_{i_0}} \bigr| \; \d \mathrm y \le C
	\frac{1}{|\kappa Q_{i_0}|}
	\int_{\kappa Q_{i_0}} \bigl| \tilde f(\mathrm y)
	- \tilde f_{\kappa Q_{i_0}} \bigr| \; \d \mathrm y,
  \intertext{since $Q_j$ and $\kappa Q_{i_0}$ are of comparable
	size and $Q_j \subseteq \kappa Q_{i_0}$. Applying the Poincar\'e
	inequality on $\kappa Q_{i_0}$, we further estimate by}
  &\le C \kappa \ell_{i_0} \frac{1}{|\kappa Q_{i_0}|}
	\int_{\kappa Q_{i_0}} \bigl| \nabla \tilde f(\mathrm y) \bigr|
	\; \dd \mathrm{y} \le C \ell_j
	\frac{1}{|\kappa Q_{i_0}|}
	\int_{\kappa Q_{i_0}} \bigl| \nabla \tilde f(\mathrm y)
	\bigr| \; \d \mathrm y.
  \intertext{Since $\kappa \ge \tilde c$, there is again some point $\mathrm z
	\in \kappa Q_{i_0} \cap F$ and we may continue as above}
  &\le C \ell_j M \bigl( |\nabla \tilde f| + |\tilde f| + |\tilde f|/\d_D
	\bigr)(\mathrm z) \le C \ell_j \alpha.
\end{align*}
Putting everything together and investing that $I(\mathrm x)$ is
uniformly finite for every $\mathrm x \in E$, we have achieved
\[ |\nabla \tilde g(\mathrm x)| \le \bigl| h_s(\mathrm x) \bigr| + \bigl| (h_u
	+ h_s) (\mathrm x) \bigr| \le C \alpha
\]
and have thus proved \ref{CZZ(2)}.

\subsection*{Step 6: Proof of \ref{CZZ(6)}}
We first estimate
\[ \| g \|_{W^{1,p}_D} \le \| \tilde g \|_{W^{1,p}_D(\R^d)} =  \Bigl\|
	\tilde f - \sum_{j \in I} \tilde b_j \Bigr\|_{W^{1,p}_D(\R^d)} \le
	\| \tilde f \|_{W^{1,p}_D(\R^d)} + \Bigl\| \sum_{j \in I} \tilde b_j
	\Bigr\|_{W^{1,p}_D(\R^d)}.
\]
By the continuity of the extension operator we have $\| \tilde f
\|_{W^{1,p}_D(\R^d)} \le C \| f \|_{W^{1,p}_D}$, so we only have to estimate
the sum of the $\tilde b_j$, $j \in I$.

Here we again rely on \ref{CZZ(5)} and the equivalence of norms in $\R^N$ to
obtain
\begin{equation} \label{e-6.1}
  \Bigl\| \sum_{j \in I} \tilde b_j \Bigr\|_{L^p(\R^d)}^p = \int_{\R^d} \Bigl|
	\sum_{j \in I} \tilde b_j \Bigr|^p \le \int_{\R^d} \Bigl(
	\sum_{j \in I} |\tilde b_j| \Bigr)^p \le C \int_{\R^d} \sum_{j \in I}
	|\tilde b_j|^p = C \sum_{j \in I} \int_{Q_j} |\tilde b_j|^p.
\end{equation}
Investing the estimates in \eqref{e-b_jwithq} for $q = p$ and in
\eqref{e-b_jspecial} for usual and special cubes, respectively, we find
\begin{equation} \label{e-6.2}
  \int_{Q_j} | \tilde b_j|^p \le C \int_{Q_j} \bigl( |\tilde f|^p +
	|\nabla \tilde f|^p \bigr).
\end{equation}
Combining the two last estimates we thus have with the help of \ref{CZZ(5)}
\[ \Bigl\| \sum_{j \in I} \tilde b_j \Bigr\|_{L^p(\R^d)}^p \le C
	\sum_{j \in I} \int_{Q_j} \bigl( |\tilde f|^p  + |\nabla \tilde f|^p
	\bigr) \le C \int_{\R^d} \sum_{j \in I} \mathbf{1}_{Q_j} \bigl(
	|\tilde f|^p + |\nabla \tilde f|^p \bigr) \le C
	\| \tilde f \|_{W^{1,p}_D(\R^d)}.
\]
For the estimate of the gradient,
we argue as in \eqref{e-6.1} and \eqref{e-6.2}, in order to find thanks to the
estimates in \eqref{e-nablabj} for $q = p$ and \eqref{e-nablabjspecial}
\[ \Bigl\| \sum_{j \in I} \nabla \tilde b_j \Bigr\|_{L^p(\R^d)}^p \le C
	\sum_{j \in I} \int_{Q_j} |\nabla \tilde b_j|^p \le C \sum_{j \in I}
	\int_{Q_j} \Bigl( |\nabla \tilde f|^p + \frac{|\tilde f|^p}{\d_D^p}
	\Bigr).
\]
Investing again \ref{CZZ(5)} and the Hardy inequality in \eqref{e-TildeHardy},
we end up with
\[ \Bigl\| \sum_{j \in I} \nabla \tilde b_j \Bigr\|_{L^p(\R^d)}^p \le C
	\int_{\R^d} \Bigl( |\nabla \tilde f|^p + \frac{|\tilde f|^p}{\d_D^p}
	\Bigr) \le C \int_{\R^d} |\nabla \tilde f|^p \le \| \tilde f
	\|_{W^{1,p}_D(\R^d)}
\]
and this finishes the proof, thanks to $\| \tilde f \|_{W^{1,p}_D(\R^d)} \le C
\| f \|_{W^{1,p}_D}$.

Having the Calder\'on-Zygmund decomposition at hand, we can now show that it
really respects the boundary condition on $D$.

\begin{corollary} \label{cor-CZZ}
Let $p \in {]1, \infty[}$ and $f \in W^{1,p}_D$ be given. The functions $g$ and
$b = \sum_{j \in I} b_j$ from Lemma~\ref{CZZ} have the following properties:
\begin{enumerate}
 \item \label{cor-CZZ:i} $b \in W^{1,1}_D$ with $\|b\|_{W^{1,1}} \le C
	\alpha^{1-p} \| f \|_{W^{1,p}_D}^p$,
 \item \label{cor-CZZ:ii} $g \in W^{1, \infty}_D$ with $\|g\|_{W^{1, \infty}_D}
	\le C \alpha$,
 \item \label{cor-CZZ:iii} If $f \in W^{1,2}_D$, then also $g,b \in W^{1,2}_D$.
\end{enumerate}
\end{corollary}

\begin{proof}
\begin{enumerate}
 \item Thanks to \ref{CZZ(3)} in Lemma~\ref{CZZ} we have $b_j \in
	W^{1,1}_D(\Omega)$ for all $j \in I$. Moreover, by the estimates in
	\ref{CZZ(3)} and \ref{CZZ(4)} of the same lemma,
	\begin{equation} \label{e-b_jW11}
	 \sum_{j \in I} \|b_j \|_{W^{1,1}} \le C \alpha \sum_{j \in I} |Q_j|
		\le C \alpha^{1 - p} \|f\|_{W^{1,p}_D}^p < \infty.
	\end{equation}
	Thus, the sum in $b$ is absolutely convergent in $W^{1,1}$, which means
	that $b$ satisfies the asserted norm estimate and lies in the closed
	subspace $W^{1,1}_D$. Thus, we have achieved \ref{cor-CZZ:i}.
 \item We first show that $\tilde g$ has a Lipschitz continuous representative
	and that the Lipschitz constant is controlled by $C\alpha$. From
	the proof of Lemma~\ref{CZZ} we have $\tilde{g} \in W^{1,p}(\R^d)$ for
	all $1 \le p < \infty$. So, from \cite[Section~2]{hajlasz4} we can
	infer that for almost all $\mathrm{x}, \mathrm{y} \in
	\R^d$
	\[ \bigl|\tilde g (\mathrm{x}) - \tilde g (\mathrm{y}) \bigr| \le C
		|\mathrm{x} - \mathrm{y}| \left( \bigl( M ( |\nabla \tilde{g}
		|^p ) \bigr)^{\frac 1p} (\mathrm{x}) + \bigl( M (
		|\nabla \tilde g|^p ) \bigr)^{\frac 1p} (\mathrm{y}) \right).
	\]
	The Hardy-Littlewood maximal operator is bounded on $L^\infty(\R^d)$,
	so this implies
	\[ \sup_{\mathrm{x}, \mathrm{y} \in \R^d, \mathrm{x} \ne \mathrm{y}}
		\frac{|\tilde g(\mathrm{x}) - \tilde g(\mathrm{y})|}
		{|\mathrm{x} - \mathrm{y}|} \le C \| \nabla \tilde g
		\|_{L^\infty(\R^d)} \le C \alpha
	\]
	and we find $\tilde g \in W^{1, \infty}(\R^d) = (L^\infty \cap
	\mathrm{Lip})(\R^d)$.

	It remains to prove the right boundary behaviour of $\tilde g$,
	i.e.\@ $\tilde g|_D = 0$. Then by the Definition of
	$W^{1,\infty}_D$, cf.\@ \eqref{e-LipRaum}, we find $g =
	\tilde g|_\Omega \in W^{1,\infty}_D$. Since $\tilde f,
	\tilde b \in W^{1,1}_D(\R^n)$, these two functions have zero trace on
	$D$ $\mathcal H_{d-1}$-almost everywhere, so the same
	is true for $\tilde g$ and we only have to get rid of the ``almost
	everywhere''. Let $\mathrm{x} \in D$ be given. Then for every $\eps >
	0$, by the Ahlfors-David condition~\eqref{e-ahlf}, we have
	$\sigma(B(\mathrm{x}, \eps) \cap D) > 0$, so there must be points in
	this set, where $\tilde g$ vanishes. But this means that $\mathrm{x}$
	is an accumulation point of the set $\{ \mathrm{y} \in D :
	\tilde g(\mathrm{y}) = 0 \}$. By the continuity of
	$\tilde g$ this implies $\tilde g (\mathrm{x}) = 0$.
 \item By \ref{cor-CZZ:ii} and Lemma~\ref{inclusion} we have $g \in
	W^{1, \infty}_D \hookrightarrow W^{1,2}_D$, so with $f$ also $b$ is in
	this space.
 \qedhere
\end{enumerate}
\end{proof}

\section{Real interpolation of the spaces $W^{1,p}_D(\Omega)$}
\label{sec-Interpolation}
%
%
%
%
%
%
%
%
\noindent
In this section we establish interpolation within the set of spaces 
$\{W^{1,p}_D(\Omega)\}_{p \in [1,\infty]}$. There already exist interpolation
results for spaces of this scale which incorporate mixed boundary conditions
(compare \cite{mitrea}, \cite{ggkr}) but -- to our knowledge -- not of the
required generality concerning the Dirichlet part. The key ingredient for this
generalization will be the Calder\'on-Zygmund decomposition proved in
Section~\ref{sec-CZD}.

\subsection{The interpolation result}
The main result of this section is the following.

\begin{theorem} \label{interp}
 Let $\Omega$ and $D$ satisfy Assumption~\ref{assu-general}. 
 Then for all choices of $1 \le p_0 < p < p_1 \le \infty$ we have for
 $\alpha = \frac{(p - p_0)p_1}{(p_1 - p_0)p}$
 \[ W^{1,p}_D(\Omega) = \bigl( W^{1, p_0}_D(\Omega), W^{1, p_1}_D(\Omega)
	\bigr)_{\alpha, p}
 \]
 with equivalent norms.
\end{theorem}

We recall the following complex reiteration theorem:

\begin{theorem} \label{cwikel}\cite{bergh,cwikel}
For any compatible couple of Banach spaces $(A_0,A_1)$ we have
\[ \bigl[ (A_0,A_1)_{\lambda_0,p_0}, (A_0,A_1)_{\lambda_1,p_1} \bigr]_\alpha
	= (A_0,A_1)_{\beta,p}
\]
for all $\lambda_0, \lambda_1$ and $\alpha$ in $(0,1)$ and all $p_0, p_1$ in
$[1,\infty]$, except for the case $p_0 = p_1 = \infty$. Here $\beta$ and $p$
are given by $\beta = (1-\alpha)\lambda_0 + \alpha \lambda_1$ and $\frac{1}{p}
= \frac{1-\alpha}{p_0} + \frac{\alpha}{p_1}$.
\end{theorem}

>From this theorem and our real interpolation Theorem~\ref{interp}, a
complex interpolation result for Sobolev spaces $W^{1,p}_D(\Omega)$ follows.

\begin{corollary} \label{complex1}
 Let $\Omega$ and $D$ satisfy Assumption~\ref{assu-general}.
 For $1 < p_0 < p < p_1 < \infty$ and $\alpha =
 \frac{\frac{1}{p_0} - \frac{1}{p}}{\frac{1}{p_0} - \frac{1}{p_1}} =
 \frac{p_1 (p - p_0)}{p (p_1 - p_0)}$, we have
 \[ \bigl[ W^{1,p_0}_D(\Omega), W^{1,p_1}_D(\Omega) \bigr]_{\alpha} =
	W^{1,p}_D(\Omega).
 \]
\end{corollary}

\subsection{The $K$-Method of real interpolation}

The reader can refer to \cite{BS88}, \cite{bergh} for details on the
development of this theory. Here we only recall the essentials to be used in
the sequel. 

Let $A_0$, $A_1$ be two normed vector spaces embedded in a topological
Hausdorff vector space $V$. For each  $a \in A_0 + A_1$ and $t > 0$, we define
the $K$-functional of interpolation by
\[ K(a, t, A_0, A_1) = \inf_{a = a_0 + a_1} \bigl( \|a_{0}\|_{A_0} + t \| a_1
	\|_{A_1} \bigr).
\]
For $0 < \theta < 1$, $1 \leq q \leq \infty$, the real interpolation space
$(A_0, A_1)_{\theta,q}$ between $A_0$ and $A_1$ is given by
\[ (A_0, A_1)_{\theta,q} = \Bigl\{ a \in A_0 + A_1 : \|a\|_{\theta,q} :=
	\Bigl( \int_0^\infty \bigl( t^{-\theta} K(a, t, A_0, A_1) \bigr)^q
	\; \frac{\dd t}{t} \Bigr)^{1/q} < \infty \Bigr\}.
\]
It is an exact interpolation space of exponent $\theta$ between $A_0$ and
$A_1$, see \cite[Chapter~II]{bergh}.


\begin{definition}
 Let $f : X \to \R$ be a measurable function on a measure space $(X,\mu)$. The
 \emph{decreasing rearrangement} of $f$ is the function $f^* : {]0, \infty[}
 \to \R$ defined by
 \[ f^*(t) = \inf \bigl\{ \lambda : \, \mu (\{ x : \,|f(x)| > \lambda \} \le
	t \bigr\}.
 \]
 The \emph{maximal decreasing rearrangement} of $f$ is the function
 $f^{**}$ defined for every $t>0$ by
 \[ f^{**}(t) = \frac{1}{t} \int_0^t f^{*}(s) \; \dd s.
 \]
\end{definition}

\begin{remark} \label{rem-int}
 It is well known that when $X$ satisfies the doubling property, then
 $(Mf)^* \le C f^{**}$, where $M$ is again the Hardy-Littlewood maximal
 operator from \eqref{e-HLMO}. This is an easy consequence of the fact that
 $M$ is of weak type $(1, 1)$ and of strong type $(\infty,\infty)$, see
 \cite[Theorem~3.8,  p.~122]{BS88}, and $\mu (\{ x : \, |f(x)| > f^*(t) \})
 \le t$ for all $t > 0$.

 We refer to \cite{BS88}, \cite{bergh} for other properties of $f^*$ and
 $f^{**}$.
\end{remark}

We conclude  by quoting the following classical result (\cite[p.~109]{bergh}):

\begin{proposition}\label{IK}
 Let $(X,\mu)$ be a measure space with a $\sigma$-finite positive measure
 $\mu$. Let $f \in L^1(X) + L^\infty(X)$. We then have
 \begin{enumerate}
  \item \label{IK:i} $K(f, t, L^1, L^\infty) = t f^{**}(t)$ and
  \item \label{IK:ii} for $1 \le p_0 < p < p_1 \leq \infty$ it holds
	$(L^{p_0}, L^{p_1})_{\theta,p} = L^p$ with equivalent norms, where
	$1/p = (1-\theta)/p_0 + \theta/p_1$ with $0 < \theta < 1$.
 \end{enumerate}
\end{proposition}

\subsection{Proof of the interpolation result}

The proof of Theorem~\ref{interp} is based on the following estimates for the
$K$-functional.

\begin{lemma} \label{CK}
 Let $1<p<\infty$. We have for all $t>0$ 
 \[ K(f, t, W^{1,1}_D, W^{1,\infty}_D) \ge C_1 t \bigl( |f|^{**}(t) +
	|\nabla f|^{**}(t) \bigr) \qquad \text{for all } f \in W^{1,1}_D +
	W^{1,\infty}_D
 \]
 and
 \[ K(f, t, W^{1,1}_D, W^{1,\infty}_D) \le C_2 t \Bigl( |\nabla \tilde f|^{**}
	(t) + |\tilde f|^{**}(t) + \Bigl( \frac{|\tilde f |}{\dd_D} \Bigr)^{**}
	(t) \Bigr) \qquad \text{for all } f \in W^{1,p}_D.
 \]
 The  constants $C_1$, $C_2$ are independent of $f$ and $t$, and $\tilde f=\mathfrak E f$
 is the Sobolev extension of $f$ from Lemma~\ref{l-extend}.
\end{lemma}

\begin{proof}
 For the lower bounds, let $f \in W^{1,1}_D + W^{1,\infty}_D$ be given. Then
 due to Proposition~\ref{IK}~\ref{IK:i}
 \begin{align*}
  K(f, t, W^{1,1}_D, W^{1, \infty}_D) &\ge \bigl( \inf_{f = f_0 + f_1}
	 ( \| f_0 \|_{L^1} + t \| f_1\|_{L^\infty} ) + \inf_{f = f_0 + f_1}
	( \| \nabla f_0 \|_{L^1} + t \| \nabla f_1 \|_{L^\infty} ) \bigr) \\
  &\ge C \bigl( K(|f|, t, L^1, L^\infty) + K(|\nabla f|, t, L^1, L^\infty)
	\bigr) = C t\bigl( |f|^{**}(t) + |\nabla f|^{**}(t) \bigr).
 \end{align*}
 Now, for the upper bound, we consider $f \in W^{1,p}_D$. For every $t > 0$ we
 set
 \[ \alpha(t) := \Bigl( M \Bigl( |\nabla \tilde f| + |\tilde f| + \Bigl|
	\frac{\tilde f}{\d_D} \Bigr| \Bigr) \Bigr)^{*}(t)
 \]
 and we recall from the proof of Lemma~\ref{CZZ} the notation
 \[ E = E_t = \Bigl\lbrace \mathrm{x} \in \R^d : M \Bigl( |\nabla \tilde f| +
	|\tilde f| + \Bigl| \frac{\tilde f}{\d_D} \Bigr| \Bigr)(\mathrm x) >
	\alpha(t) \Bigr\rbrace.
 \]
 Remark that with this choice of $\alpha(t)$, we have $|E_t| \le t$ for all
 $t > 0$. Furthermore, due to Remark~\ref{rem-int} applied with $X = \R^d$
 \begin{equation} \label{e-alphat}
  \alpha(t) \le C \Bigl( |\nabla \tilde f|^{**} + |\tilde f|^{**} + \Bigl|
	\frac{\tilde f}{\d_D} \Bigr|^{**} \Bigr)(t).
 \end{equation}
    
 Now, we take the Calder\'on-Zygmund decomposition from Lemma~\ref{CZZ} for $f$
 with this choice of $\alpha(t)$. This results in a decomposition of
 $f\in W^{1,p}_D$ as $f = g + b$ with $b \in W^{1,1}_D$ and $g \in
 W^{1,\infty}_D$. Invoking Corollary~\ref{cor-CZZ}~\ref{cor-CZZ:ii}, we have
 $\|g\|_{W^{1,\infty}_D} \le C \alpha(t)$ and from \eqref{e-b_jW11} we deduce
 \[ \| b \|_{W^{1,1}_D} \le C \alpha(t) \sum_{j \in I} |Q_j| \le C \alpha(t)
	|E_t| \le C t \alpha(t).
 \]
 Combining these estimates with \eqref{e-alphat}, we find
 \[ K(f, t, W^{1,1}_D, W^{1,\infty}_D) \le \|b\|_{W^{1,1}_D} + t
	\|g\|_{W^{1,\infty}_D} \le C t \alpha(t) \le C t \Bigl( |\nabla
	\tilde f|^{**}(t) + |\tilde f|^{**}(t) + \Bigl( \frac{|\tilde f|}{\d_D}
	\Bigr)^{**}(t) \Bigr).
 \]
 for all $f \in W^{1,p}_D$ and for all $t > 0$ and this was the claim.
\end{proof}

\begin{proof}[Proof of Theorem~\ref{interp}] 
 By the reiteration Theorem (cf. \cite[Thm.1.10.2]{triebel})
 it suffices to establish the special
 case of $p_0 = 1$ and $p_1 = \infty$, i.e.\@ $W^{1,p}_D = (W^{1,1}_D,
 W^{1,\infty}_D)_{1-1/p,p}$ with equivalent norms for $1 < p < \infty$.
 First, since $\Omega$ is bounded we have $W^{1,p}_D \hookrightarrow W^{1,1}_D
 \hookrightarrow W^{1,1}_D + W^{1,\infty}_D$. Moreover, for $f \in W^{1,p}_D$
 we have due to Lemma~\ref{CK}
 \begin{align*}
  \|f\|_{1-1/p,p} &= \Bigl( \int_0^\infty \bigl[ t^{1/p - 1} K(f, t,
	W^{1,1}_D, W^{1, \infty}_D) \bigr]^p \frac{\d t}{t} \Bigr)^{1/p} \\
  &\le C \Bigl( \int_0^\infty \Bigl[ t^{1/p} \Bigl( |\nabla \tilde f|^{**}(t)
	+ |\tilde f|^{**}(t) + \Bigl( \frac{|\tilde f|}{\d_D} \Bigr)^{**}(t)
	\Bigr) \Bigr]^p \frac{\d t}{t} \Bigr)^{1/p} \\
  &= C \Bigl\| |\nabla \tilde f|^{**} + |\tilde f|^{**} + \Bigl(
	\frac{|\tilde f|}{\d_D} \Bigr)^{**} \Bigr\|_{L^p(\R_+)}.
 \intertext{Since $\|g^{**}\|_{L^p(\R_+)} \sim \|g^*\|_{L^p(\R_+)} =
	\|g\|_{L^p}$, this allows us to continue}
  &\le C \Bigl( \| \nabla \tilde f \|_{L^p(\R^d)} + \|\tilde f\|_{L^p(\R^d)}
	+ \Bigl\| \frac{\tilde f}{\d_D} \Bigr\|_{L^p(\R^d)} \Bigr) \\
  &\le C \|\tilde f\|_{W^{1,p}_D(\R^d)} \le C \| f \|_{W^{1,p}_D}
\end{align*}
thanks to the Hardy inequality in \eqref{e-TildeHardy} and the continuity of
the extension operator that assigns $\tilde f$ to $f$.

Conversely, let $f\in (W^{1,1}_D, W^{1,\infty}_D)_{1-1/p,p}$. Then, invoking
the lower estimate in Lemma~\ref{CK} we find as above, investing that $g
\mapsto g^{**}$ is sublinear,
\begin{align*}
  \|f\|_{1-1/p,p} &\ge C \Bigl( \int_0^\infty \bigl[ t^{1/p} \bigl( |f|^{**}(t)
	+ |\nabla f|^{**}(t) \bigr) \bigr]^p \frac{\d t}{t} \Bigr)^{1/p} = C
	\bigl\| |f|^{**} + | \nabla f |^{**} \bigr\|_{L^p(\R_+)} \\
  &\ge C \bigl\| (|f| + |\nabla f|)^{**} \bigr\|_{L^p(\R_+)} \ge C \| |f| +
	|\nabla f| \|_{L^p} \ge C \| f \|_{W^{1,p}}.
\end{align*}
It remains to check the right boundary behaviour of $f$, i.e.\@ $f \in
W^{1,p}_D$. In order to do so, we use the fact that $W^{1,1}_D \cap
W^{1,\infty}_D$ is dense in $(W^{1,1}_D, W^{1,\infty}_D)_{1-1/p,p}$, see
\cite[Theorem~3.4.2]{bergh}.
If $f = \lim_{n \to \infty} f_n$ for some sequence $(f_n)$ in $W^{1,1}_D \cap
W^{1,\infty}_D$, then the limit is also in $W^{1,p}(\Omega)$ by the above
inequality. As $W^{1,\infty}_D \subseteq W^{1,p}_D$ by Lemma~\ref{inclusion},
we have $f_n \in W^{1,p}_D$ for every $n \in \N$. As this space is closed in
$W^{1,p}$, this yields $f \in W^{1,p}_D$ and we find
\[ \| f\|_{W^{1,p}_D} = \| f \|_{W^{1,p}} \le C \|f\|_{1-1/p,p}.
\qedhere
\]
\end{proof}


%
%
%
%
%
%
%
%
\section{Off-diagonal estimates} \label{sec-OffD}
%
%
%
%
%
%
%
%
\noindent
As a next preparatory step towards the proof of Theorem~\ref{t-mainsect}, we
show that the Gaussian estimates imply $L^p$-$L^2$ off-diagonal estimates for
the operators $T(t) := \e^{-tA_0}$ and $t A_0 T(t)$.
%
\begin{lemma}\label{offdiaglemma}
Let $p \in {[1,2]}$ and let $E, F \subseteq \Omega$ be relatively closed. Then
there exist constants $C \geq 0$ and $c > 0$, such that for
every $h \in L^2 \cap L^p$ with $\supp(h) \subseteq E$ we have for all $t > 0$
\begin{enumerate}
 \item $\displaystyle \|T(t) h \|_{L^2(F)} \le C t^{(d/2 - d/p)/2} \e^{-c \frac{\dd(E,F)^2}{t}}
	\| h\|_{L^p}$ for $p \ge 1$ and
 \item $\displaystyle \|t A_0 T(t) h \|_{L^2(F)} \le C t^{(d/2 - d/p)/2}
	\e^{-c \frac{\dd(E,F)^2}{t}} \| h\|_{L^p}$ for $p > 1$.
\end{enumerate}
\end{lemma}
%
\begin{proof}
\begin{enumerate}
 \item We denote the kernel of $T(t)$ by $k_t$. Since $A_0 = - \nabla \cdot \mu
	\nabla + 1$, using the notation of Proposition~\ref{t-gausss}, we have
	$k_t = \e^{-t} K_t$. Thus for $k_t$ we have the Gaussian estimates
	\[ 0 \le k_t(\mathrm x, \mathrm y) \le \frac{C}{t^{d/2}}
		\e^{-c \frac{|\mathrm x- \mathrm y|^2}{t}}, \quad t > 0,
		\text{ a.a. } \mathrm x, \mathrm y \in \Omega,
	\]
	without the term $\e^{\varepsilon t}$. Using these, a straightforward
	calculation shows
	\[ \|T(t) h \|_{L^2(F)}^2 \le \frac{C}{t^d}
		\e^{- c \frac{\dd(E,F)^2}{t}} \bigl\|
		\e^{- c \frac{|\cdot|^2}{2t}} * |\tilde h|
		\bigr\|^2_{L^2(\R^d)},
	\]
	where we denoted by $\tilde h$ the extension by $0$ of $h$ to the whole
	of $\R^d$. Now, applying Young's inequality to bound the convolution
	one obtains the assertion.
 \item In a first step, we observe that it is enough to show the assertion in
	the case $p=2$. In fact, we have by the first part of the proof (set $E
	= F = \Omega$ and $p=1$)
	\begin{align*}
	  \| t A_0 T(t) h \|_{L^2(F)} &\le \| T(t/2) t A_0 T(t/2) h \|_{L^2}
		\le C t^{- d/4} \| t A_0 T(t/2) h \|_{L^1} \\
	  &\le C t^{-d/4} \|h\|_{L^1},
	\end{align*}
	since $T(t)$ extrapolates to an analytic semigroup on $L^1$ by the
	Gaussian estimates, cf.\@ \cite{Hie96} or \cite{Are97}. Admitting the
	assertion in the case $p=2$:
	\[ \| t A_0 T(t) h \|_{L^2(F)} \le C \e^{-c \frac{\dd(E,F)^2}{t}}
		\|h\|_{L^2},
	\]
	the result then follows by interpolation using the Riesz-Thorin
	Theorem.

	In order to prove the off-diagonal bounds in the case $p=2$, we apply
	Davies' trick, following the proof of \cite[Proposition~2.1]{auschmem}.
	Since this procedure is rather standard, we just give the major steps.

	For some Lipschitz continuous function $\varphi : \Omega \to \R$ with
	$\|\nabla \varphi\|_{L^\infty} \le 1$ and $\varrho > 0$ we define the
	twisted form
	\[ a_\varrho (u,v) = \int_\Omega \bigl( \mu \nabla (
		\e^{-\varrho \varphi} u) \cdot \nabla (\e^{\varrho \varphi}
		\overline{v}) + u \overline{v} \bigr) \; \dd \mathrm x
		, \qquad u,v \in D(a_\varrho) := W^{1,2}_D.
	\]
	Setting $\kappa := 2 \varrho^2 \|\mu\|_{L^\infty}$ and estimating the
	real and imaginary part of the quadratic form $a_\varrho + \kappa - 1$
	one finds that the numerical range of $a_\varrho + \kappa$ lies in the
	(shifted) sector $\mathcal S + 1$, where $\mathcal S := \bigl\{ \lambda
	\in \C : |\Im \lambda | \le
	\sqrt{\frac{\|\mu\|_{L^\infty}}{\mu_\bullet}} \Re \lambda \bigr\}$ and
	$\mu_\bullet$ is the ellipticity constant from
	Assumption~\ref{assu-coeff}.

	In the following, we denote by $A_\varrho$ the operator associated to
	the form $a_\varrho$ in $L^2$. Since $A_\varrho + \kappa - 1$ is
	maximal accretive, cf.\@ \cite[Ch.~VI.2]{kato}, its negative generates
	an analytic $C_0$-semigroup $\e^{-tA_\varrho}$ on $L^2$ and $A_\varrho$
	even admits a bounded $H^\infty$-calculus there,
	cf.\@ \cite[Ch.~2.4]{DHP03} or \cite{Haa06}. Applying the functional
	calculus of $A_\varrho$, for every $t \ge 0$ we find
	\begin{align}
	  \bigl\| t A_\varrho \e^{-t A_\varrho} \bigr\| &\le
		\bigl\| t (A_\varrho + \kappa) \e^{-t( A_\varrho + \kappa)}
		\bigr\| \e^{t \kappa} + \bigl\| \e^{-t (A_\varrho + \kappa)}
		\bigr\| t \kappa \e^{t \kappa} \nonumber\\
	  &\le C \e^{t \kappa} + C \e^{2 t \kappa} \le C
		\e^{4 \varrho^2 t \|\mu\|_{L^\infty}}. \label{e-varrhoAbsch}
	\end{align}
	Recalling that the form domain $W^{1,2}_D$ is invariant under
	multiplications with $\e^{\varrho \varphi}$ by
	Proposition~\ref{p-basicl2}~\ref{p-basicl2:iii}, it is easy to verify
	that for every $f \in L^2$ with $\e^{-\varrho \varphi} f \in D(A_0)$,
	we have $A_\varrho f = - \e^{\varrho \varphi} A_0
	\e^{- \varrho \varphi} f$. From this we then deduce
	\[ R(\lambda, A_\varrho) = \e^{\varrho \varphi} R(\lambda, A_0)
		\e^{- \varrho \varphi}, \quad \text{for all } \lambda >
		\varrho^2 \|\mu\|_{L^\infty},
	\]
	which finally yields for every $f \in L^2$
	\[ \e^{-t A_\varrho} f = \lim_{n \to \infty} \Bigl[ \frac nt R(n/t,
		A_\varrho) \Bigr]^n f = \e^{\varrho \varphi}
		\lim_{n \to \infty} \Bigl[ \frac nt R(n/t, A_0) \Bigr]^n
		\e^{- \varrho \varphi} f = \e^{\varrho \varphi} T(t)
		\e^{-\varrho \varphi} f.
	\]
	Now we specify $\varphi(\mathrm x) = \dd(\mathrm x, E)$ for $\mathrm x
	\in \Omega$. Then for every $h \in L^2$ with support in $E$ and all
	$\varrho, t > 0$ we get
	\[ t A_0 T(t) h = - t \frac{\dd}{\dd t} T(t) h = t
		\e^{-\varrho \varphi} A_\varrho \e^{-t A_\varrho}
		\e^{\varrho \varphi} h = t \e^{- \varrho \varphi} A_\varrho
		\e^{-t A_\varrho} h,
	\]
	as $\varphi = 0$ on the support of $h$. This yields for all $\varrho, t
	> 0$
	\begin{align*}
	  \| t A_0 T(t) h\|_{L^2(F)} &= \| t \e^{-\varrho \dd(\cdot, E)}
		A_\varrho \e^{-t A_\varrho} h\|_{L^2(F)} \le
		\e^{-\varrho \dd(E, F)} \| t A_\varrho \e^{- t A_\varrho} h
		\|_{L^2} \\
	  &\le C \e^{4\varrho^2 \|\mu\|_{L^\infty} t - \varrho \dd(E,F)}
		\|h\|_{L^2},
	\end{align*}
	thanks to \eqref{e-varrhoAbsch}. Minimizing over $\varrho > 0$ finally
	yields the assertion with $c = (8 \|\mu\|_{L^\infty})^{-1}$.
  \qedhere
 \end{enumerate}
\end{proof}

%
%
%
%
%
%
%
%
%

%
%
%
%
\section{Proof of the main result} \label{sec-Proof}
%
%
%
%
\noindent
We now turn to the proof of Theorem~\ref{t-mainsect}. Building on the
hypotheses that the assertion is true for $p=2$,
cf.\@ Assumption~\ref{assu-coeffi}, we will show the corresponding inequality
in a weak~$(p,p)$ setting for all $1 < p < 2$. Then our result follows by
interpolation. More precisely we want to show the following.
%
\begin{proposition}
 Let $\Omega$ and $D$ satisfy Assumption~\ref{assu-general}, and let $\mu$ be such that
Assumptions~\ref{assu-coeff} and \ref{assu-coeffi} are true. Then there is a constant
$C \ge 0$, such that for all $p \in {]1,2[}$, for every $f \in C^\infty_D$ and all $\alpha > 0$ we have
 \begin{equation} \label{e-weak}
  \bigl| \bigl\{ \mathrm x \in \Omega : |A_0^{1/2} f (\mathrm x) | > \alpha
	\bigr\} \bigr| \le \frac{C}{\alpha^p} \| f \|_{W^{1,p}_D}^p.
 \end{equation}
\end{proposition}
%
\begin{proof}
 We follow the proof of \cite[Lemma~4.13]{auschmem}. Let $\alpha > 0$, $p \in
 {]1,2[}$ and $f \in C^\infty_D$ be given. We apply the Calder\'on-Zygmund
 decomposition from Lemma~\ref{CZZ} to write $f = g + \sum_{j \in I} b_j$. In
 all what follows the references \ref{CZZ(1)} -- \ref{CZZ(6)} will stand for
 the corresponding features in Lemma~\ref{CZZ}.

 Since $C^\infty_D \hookrightarrow W^{1,2}_D =
 \mathrm{dom}_{L^2}(A_0^{1/2})$, by Corollary~\ref{cor-CZZ}~\ref{cor-CZZ:iii}
 also the functions $g$ and $b = \sum_{j \in I} b_j$ are in the $L^2$-domain of
 $A_0^{1/2}$ and $A_0^{1/2} b = \sum_{j \in I} A_0^{1/2} b_j$. Thus, we can estimate
 \begin{equation} \label{I,II}
  \bigl| \bigl\{ \mathrm x \in \Omega : |A_0^{1/2} f (\mathrm x) | > \alpha
	\bigr\} \bigr| \le \Bigl| \Bigl\{ \mathrm x \in \Omega : \bigl|
	A_0^{1/2} g(\mathrm x) \bigr| > \frac \alpha2 \Bigr\} \Bigr| + \Bigl|
	\Bigl\{ \mathrm x \in \Omega : \Bigl| \Bigl(A_0^{1/2} b \Bigr)
	(\mathrm x) \Bigl| > \frac \alpha2 \Bigr\} \Bigr|,
 \end{equation}
and our aim is to bound both terms on the right hand side by $C \| f
\|_{W^{1,p}_D}^p/\alpha^p$.

The one containing $g$ is as always the easy part. We first note
that thanks to \ref{CZZ(6)} and Corollary~\ref{cor-CZZ} we
know
\[ \| g \|_{W^{1,p}_D} \le C \| f\|_{W^{1,p}_D} \qquad \text{and} \qquad
	\| g \|_{W^{1, \infty}_D} \le C \alpha.
\]
By interpolation this yields
\[ \| g \|_{W^{1,2}_D}^2 \le C \| g \|_{W^{1,p}_D}^p \| g
	\|_{W^{1,\infty}_D}^{2 - p} \le C \alpha^{2 - p} \| f \|_{W^{1,p}_D}^p.
\]
%
This implies, using the Tchebychev inequality and Assumption~\ref{assu-coeffi}
\[ \Bigl| \Bigl\{ \mathrm x \in \Omega : \bigl| A_0^{1/2} g (\mathrm x) \bigr|
	> \frac \alpha2 \Bigr\} \Bigr| \le \frac{C}{\alpha^2} \| A_0^{1/2} g
	\|_{L^2}^2 \le \frac{C}{\alpha^2} \| g \|_{W^{1,2}_D}^2 \le
	\frac{C}{\alpha^p} \| f \|_{W^{1,p}_D}^p.
\]
Let's turn to the estimate of the second part in \eqref{I,II}.
We first recall the integral representation of the square root
\[ A_0^{1/2} u = \frac{2}{\sqrt{\pi}} \int_0^\infty A_0 \e^{-t^2 A_0} u
	\; \dd t \quad \text{for all } u \in \mathrm{dom}_{L^2}(A_0^{1/2}),
\]
which can be deduced straightforwardly from the well known formula (see
\cite[Ch.~2.6]{pazy})
\[ A_0^{-1/2}=\frac {1}{\sqrt \pi} \int_0^\infty \frac {e^{-tA_0}}{\sqrt t}
	\; \dd t.
\]
This yields thanks to $A_0^{1/2} b = \sum_{j \in I} A_0^{1/2} b_j$
\begin{align*}
  \Bigl| \Bigl\{ \mathrm x \in \Omega : \Bigl| \Bigl(A_0^{1/2} b \Bigr)
	(\mathrm x) \Bigr| > \frac \alpha2 \Bigr\} \Bigr| &= \Bigl| \Bigl\{
	\mathrm x \in \Omega : \Bigl| \frac{2}{\sqrt \pi}
	\int_0^\infty \sum_{j \in I} \Bigl( A_0 \e^{-t^2 A_0} b_j \Bigr)
	(\mathrm x) \; \dd t \Bigr| > \frac \alpha2 \Bigr\} \Bigr| \\
  &= \limsup_{m \to \infty} \Bigl| \Bigl\{ \mathrm x \in \Omega : \Bigl|
	\frac{2}{\sqrt \pi} \int_{2^{-m}}^\infty \sum_{j \in I} \Bigl( A_0
	\e^{-t^2 A_0} b_j \Bigr) (\mathrm x) \; \dd t \Bigr| > \frac \alpha2
	\Bigr\} \Bigr|.
\end{align*}
%
In the following we denote again by $\ell_j$ the sidelength of the cube $Q_j$,
$j \in I$, and we set $r_j := 2^k$ for that value of $k \in \Z$, such that
$2^k \le \ell_j < 2^{k+1}$. With this notation we split the integral for every
$m \in \N$:
%
%
\begin{align}
  &\Bigl| \Bigl\{ \mathrm x \in \Omega : \Bigl| \frac{2}{\sqrt \pi}
	\int_{2^{-m}}^\infty \sum_{j \in I} \Bigl( A_0 \e^{-t^2 A_0} b_j
	\Bigr) (\mathrm x) \; \dd t \Bigr| > \frac \alpha2 \Bigr\} \Bigr|
	\nonumber\\
  \le\strut& \Bigl| \Bigl\{ \mathrm x \in \Omega : \Bigl| \sum_{j \in I}
        \int_{2^{-m}}^{r_j \vee 2^{-m}} A_0 \e^{-t^2 A_0} b_j (\mathrm x)
	\; \dd t \Bigr| > \frac{\sqrt{\pi}\alpha}{8} \Bigr\} \Bigr| \nonumber\\
  & \qquad + \Bigl| \Bigl\{ \mathrm x \in \Omega : \Bigl| \sum_{j \in I}
	\int_{r_j \vee 2^{-m}}^\infty A_0 \e^{-t^2 A_0} b_j (\mathrm x)
	\; \dd t \Bigr| > \frac{\sqrt{\pi}\alpha}{8} \Bigr\} \Bigr|.
	\label{e-split}
\end{align}
%
%
For the estimate of the first integral we may restrict ourselves to the case
$r_j > 2^{-m}$, since otherwise there is no contribution from this term. We do
the usual trick to split off the union of the sets $4 Q_\iota$, $\iota \in I$,
that does not produce any sort of problem due to
\[ \Bigl| \bigcup_{\iota \in I} 4 Q_\iota \Bigr| \le \sum_{\iota \in I} |4
	Q_\iota| \le C \sum_{\iota \in I} |Q_\iota|
	\stackrel{\ref{CZZ(4)}}{\le} \frac{C}{\alpha^p} \| f \|_{W^{1,p}_D}^p.
\]
So, we only have to estimate
\begin{align}
  & \Bigl| \Bigl\{ \mathrm x \in \Omega \setminus \bigcup_{\iota \in I} 4
	Q_\iota : \Bigl| \sum_{j \in I} \int_{2^{-m}}^{r_j} A_0 \e^{-t^2 A_0}
	b_j (\mathrm x) \; \dd t \Bigr| > \frac{\sqrt{\pi}\alpha}{8} \Bigr\}
	\Bigr| \nonumber \\
  =\strut& \Bigl| \Bigl\{ \mathrm x \in \Omega : \Bigl|
	{\mathbf{1}}_{(\cup_{\iota \in I} 4 Q_\iota)^c} \sum_{j \in I}
	\int_{2^{-m}}^{r_j} A_0 \e^{-t^2 A_0} b_j (\mathrm x) \; \dd t \Bigr| >
	\frac{\sqrt{\pi} \alpha}{8} \Bigr\} \Bigr|. \nonumber
  \intertext{By the Tchebychev inequality we get}
  \le\strut& \frac{C}{\alpha^2} \Bigl\|
        {\mathbf{1}}_{(\cup_{\iota \in I} 4 Q_\iota)^c} \sum_{j \in I}
	\int_{2^{-m}}^{r_j} A_0 \e^{-t^2 A_0} b_j \; \dd t \Bigr\|_{L^2}^2.
	\label{IAbschRueckspr}
\end{align}
In order to estimate this norm we take $u \in L^2(\Omega)$ with $\| u \|_{L^2}
= 1$. Then
\begin{align}
  & \Bigl| \int_\Omega u {\mathbf{1}}_{(\cup_{\iota \in I} 4 Q_\iota)^c}
        \sum_{j \in I} \int_{2^{-m}}^{r_j} A_0 \e^{-t^2 A_0} b_j \; \dd t
	\Bigr| \le \sum_{j \in I} \int_\Omega |u|
	{\mathbf{1}}_{(\cup_{\iota \in I} 4 Q_\iota)^c} \Bigl|
	\int_{2^{-m}}^{r_j} A_0 \e^{-t^2 A_0} b_j \; \dd t \Bigr|. \nonumber
  \intertext{We now split the integration over $\Omega$ into frame-like pieces
	and apply the Cauchy-Schwarz inequality. Note that the characteristic
	function results in the sum over $l$ starting only at $l = 2$.}
  \le\strut& \sum_{j \in I} \sum_{l=2}^\infty
	\int_{(2^{l+1} Q_j \setminus 2^l Q_j) \cap \Omega} |u| \Bigl|
	\int_{2^{-m}}^{r_j} A_0 \e^{-t^2 A_0} b_j \; \dd t \Bigr|
	\label{IRueckspr} \\
  \le\strut& \sum_{j \in I} \sum_{l=2}^\infty \bigl\| u
	\bigr\|_{L^2((2^{l+1} Q_j \setminus 2^l Q ) \cap \Omega)} \Bigl\|
	\int_{2^{-m}}^{r_j} A_0 \e^{-t^2 A_0} b_j \; \dd t
	\Bigr\|_{L^2((2^{l+1} Q_j \setminus 2^l Q ) \cap \Omega)}. \nonumber
\end{align}
In order to estimate the first factor of the last expression, we identify $u$
with its trivial extension by zero to $\R^d$. Then we let appear the maximal
operator to obtain for every $\mathrm y \in Q_j$
\[ \bigl\| u \bigr\|_{L^2((2^{l+1} Q_j \setminus 2^l Q ) \cap \Omega)}^2 \le
	\int_{2^{l+1}Q_j} |u|^2 \le C 2^{d(l+1)} |Q_j| \frac{1}{|2^{l+1}Q_j|}
	\int_{2^{l+1}Q_j} |u|^2 \le C 2^{dl} \ell_j^d \bigl[ M (|u|^2) \bigr]
	(\mathrm y).
\]
Applying the off-diagonal estimates for $t^2 A_0 \e^{- t^2 A_0}$ from
Lemma~\ref{offdiaglemma} with the set $Q_j \cap \Omega$ as $E$, $(2^{l+1}Q_j
\setminus 2^l Q_j) \cap \Omega$ as $F$, $d /(d-1)$ as $p$ and $b_j$ as $h$, we
get
\begin{align*}
  \bigl\| A_0 \e^{-t^2 A_0} b_j \bigr\|_{L^2((2^{l+1}Q_j \setminus 2^l Q_j)
	\cap \Omega)} &\le \frac{C}{t^2}t^{d/2 - (d - 1)}
	\e^{- c \frac{\dd(E,F)^2}{t^2}} \|b_j\|_{L^{d/(d-1)}} \\
  &\le \frac{C}{t^{1 + d/2}} \e^{- c \frac{4^l r_j^2}{t^2}}
	\|b_j\|_{L^{d/(d-1)}},
\end{align*}
since $\dd(E,F) \ge \dd(Q_j, 2^{l+1}Q_j \setminus 2^lQ_j) \ge c (2^l \ell_j -
\ell_j) \ge c (2^l - 1) r_j \ge c 2^l r_j$
thanks to $l \ge 2$.

According to \ref{CZZ(3)} the functions $b_j$ are from $W^{1,1}_D$. 
Exploiting the Sobolev embedding $W_D^{1,1} \hookrightarrow L^{d/(d-1)}$
 (cf.\@ Remark~\ref{r-remain}~\ref{r-remain:ii})
\begin{equation} \label{PoincSobVerweis}
  \|b_j\|_{L^{d/(d-1)}} \le C \| b_j \|_{W^{1,1}} \le C \alpha
	|Q_j| \le C \alpha \ell_j^d.
\end{equation}
Putting all this together we find for our second factor
\begin{align*}
  \Bigl\| \int_{2^{-m}}^{r_j} A_0 \e^{- t^2 A_0} b_j \; \dd t
	\Bigr\|_{L^2((2^{l+1}Q_j \setminus 2^l Q_j)\cap \Omega )} &\le
	\int_{2^{-m}}^{r_j} \| A_0 \e^{- t^2 A_0} b_j
	\|_{L^2((2^{l+1}Q_j \setminus 2^l Q_j)\cap \Omega)} \; \dd t \\
  &\le C \alpha \ell_j^d \int_{2^{-m}}^{r_j} \frac{1}{t^{1 + d/2}}
	\e^{- c \frac{4^l r_j^2}{t^2}} \; \dd t \displaybreak[0]\\
  &= C \alpha \ell_j^d \int_{c 4^l}^{c 4^l r_j^2 4^{m}} \Bigl(
	\frac{\sqrt{s}}{2^l r_j} \Bigr)^{1 + d/2}
	\e^{-s} 2^l r_j s^{-3/2} \; \dd s \displaybreak[0] \\
  &\le C \alpha \ell_j^d r_j^{-d/2} 2^{-ld/2} \int_{c 4^l}^\infty s^{- 1 + d/4}
	\e^{-s} \; \dd s,
  \intertext{which is now independent of $m \in \N$. Since the integrand is
	positive and $r_j \ge 2\ell_j$ we may continue}
  &\le C \alpha \ell_j^{d/2} 2^{-ld/2} \e^{-c 4^l} \int_{c 4^l}^\infty
	s^{- 1 + d/4} \e^{-s + c4^l} \; \dd s \displaybreak[0] \\
  &= C \alpha \ell_j^{d/2} 2^{-ld/2} \e^{-c 4^l} \int_0^\infty
        (\sigma + c4^l)^{- 1 + d/4} \e^{-\sigma} \; \dd \sigma
	\displaybreak[0]\\
  &= C \alpha \ell_j^{d/2} 4^{-l} \e^{-c4^l} \int_0^\infty
        (\sigma 4^{-l} + c)^{- 1 + d/4} \e^{-\sigma} \; \dd \sigma.
\end{align*}
This last integral is bounded uniformly in $l \ge 2$. In fact, if $d > 4$, then
we estimate $4^{-l} \le 4^{-2}$ and if $d \le 4$, we may just estimate by
dropping out the whole $\sigma 4^{-l}$. So, estimating once more $4^{-l} \le
4^{-2}$, we end up with
\[ \Bigl\| \int_{2^{-m}}^{r_j} A_0 \e^{- t^2 A_0} b_j \; \dd t
        \Bigr\|_{L^2((2^{l+1}Q_j \setminus 2^l Q_j) \cap \Omega)} \le C \alpha
	\ell_j^{d/2} \e^{-c 4^l}.
\]
Coming back to \eqref{IRueckspr} we thus have
\[ \int_{(2^{l+1} Q_j \setminus 2^l Q_j) \cap \Omega} |u| \Bigl|
	\int_{2^{-m}}^{r_j} A_0 \e^{-t^2 A_0} b_j \; \dd t \Bigr| \le C
	2^{ld/2} \ell_j^{d/2} \bigl( \bigl[ M(|u|^2) \bigr] (\mathrm y)
	\bigr)^{1/2} \alpha \ell_j^{d/2} \e^{-c 4^l}
\]
for every $\mathrm y \in Q_j$. Averaging over $\mathrm y$ the inequality
remains valid and we get
\begin{align*}
  & \sum_{j \in I} \sum_{l=2}^\infty
	\int_{(2^{l+1} Q_j \setminus 2^l Q_j) \cap \Omega} |u| \Bigl|
	\int_{2^{-m}}^{r_j} A_0 \e^{-t^2 A_0} b_j \; \dd t \Bigr| \\
  \le\strut& C \sum_{j \in I} \sum_{l=2}^\infty \frac{1}{|Q_j|} \int_{Q_j}
	\alpha 2^{ld/2} \ell_j^d \e^{-c4^l} \bigl( \bigl[ M (|u|^2) \bigr]
	(\mathrm y) \bigr)^{1/2} \; \dd \mathrm y \displaybreak[0] \\
  \le\strut& C \alpha \sum_{j \in I} \sum_{l=2}^\infty 2^{ld/2} \e^{-c4^l}
        \int_{Q_j} \bigl( \bigl[ M (|u|^2) \bigr](\mathrm y) \bigr)^{1/2}
	\; \dd \mathrm y. \displaybreak[0]
\intertext{The sum over $l$ now turns out to be convergent, so we continue}
  \le\strut& C \alpha \int_{\R^d} \sum_{j \in I} {\bf 1}_{Q_j}(\mathrm y)
	\bigl( \bigl[ M (|u|^2) \bigr](\mathrm y) \bigr)^{1/2}
	\; \dd \mathrm y   \le\strut C \alpha \int_{\bigcup_{j\in I} Q_j} \bigl( \bigl[ M (|u|^2)
        \bigr](\mathrm y) \bigr)^{1/2} \; \dd \mathrm y,
\end{align*}
where we used \ref{CZZ(5)} in the last step. By the Kolmogorov inequality
(cf.\@ \cite[IV.7.19]{Tor86}) we have
\[ \int_{\bigcup_{j\in I} Q_j} \bigl( \bigl[ M (|u|^2) \bigr](\mathrm y)
	\bigr)^{1/2} \; \dd \mathrm y \le C \Bigl| \bigcup_{j \in I} Q_j
	\Bigr|^{1/2} \| |u|^2 \|_{L^1(\R^d))}^{1/2} \le C \Bigl(
	\sum_{j \in I} |Q_j| \Bigr)^{1/2} \|u\|_{L^2}.
\]
Coming back to \eqref{IAbschRueckspr}, we thus finally achieve (observe that
$\|u\|_{L^2} = 1$)
\begin{align*}
  & \Bigl| \Bigl\{ \mathrm x \in \Omega \setminus \bigcup_{\iota \in I} 4
	Q_\iota : \Bigl| \sum_{j \in I} \int_{2^{-m}}^{r_j} A_0 \e^{-t^2 A_0}
	b_j (\mathrm x) \; \dd t \Bigr| > \frac{\sqrt{\pi}\alpha}{8} \Bigr\}
	\Bigr| \\
  \le\strut& \frac{C}{\alpha^2} \Bigl\|
        {\bf 1}_{(\bigcup_{\iota \in I} 4 Q_\iota)^c} \sum_{j \in I}
	\int_{2^{-m}}^{r_j} A_0 \e^{- t^2 A_0} b_j \; \dd t \Bigr\|_{L^2}^2
	\le C \sum_{j \in I} |Q_j| \le \frac{C}{\alpha^p} \|f\|_{W^{1,p}_D}^p
\end{align*}
by \ref{CZZ(4)}.

We turn to the estimate of the second addend on the right hand side of
\eqref{e-split}. For this task, we will again need the notion of a bounded
$H^\infty$-calculus. The definition and further information can be
found in \cite{DHP03} or \cite{Haa06}.

We define the function
\[ \psi(z) := \int_1^\infty z \e^{-t^2 z} \; \dd t, \quad \Re(z) > 0.
\]
We show that
\[ \psi \in {\mathcal H}^\infty_0(\Sigma_\mu) := \Bigl\{ f : \Sigma_\mu \to \C
	\text{ analytic and } \exists \eps > 0 \text{ s.t. } |f(z)| \le C
	\frac{|z|^\eps}{(1 + |z|)^{2\eps}} \text{ for all } z \in \Sigma_\mu
	\Bigr\}
\]
for every $\mu \in {]0, \pi/2[}$, where $\Sigma_\mu := \{ z \in \C : |\arg(z)|
< \mu \}$. In fact we have substituting $\tau = t^2 \Re(z) - \Re(z)$
\begin{align*}
  \Bigl| \frac{(1 + |z|)^{2\epsilon}}{|z|^\epsilon} \psi(z) \Bigr| &\le
        \int_1^\infty |z|^{1 - \epsilon} (1 + |z|)^{2 \epsilon}
        \e^{- t^2 \Re (z) } \; \dd t \\
  &= \int_0^\infty |z|^{1 - \epsilon} (1 + |z|)^{2 \epsilon} \e^{- \tau}
        \e^{- \Re(z)} \frac{1}{2 \sqrt{\Re(z)(\tau + \Re(z))}} \; \dd \tau \\
  &\le C |z|^{1/2 - \epsilon} (1 + |z|)^{2 \epsilon} \e^{- c |z|} \int_0^\infty
        \frac{\e^{-\tau}}{\sqrt \tau} \; \dd \tau,
\end{align*}
since $\Re(z) \sim |z|$, thanks to $|\arg(z)| < \mu < \pi/2$. Thus, we may
choose $\epsilon \in {]0, 1/2[}$.

Furthermore, we have for every $z\in \C$ with $\Re(z) > 0$ and every $r > 0$
\[ \frac{1}{r} \psi(r^2 z) = \int_{r}^\infty z \e^{-t^2 z} \; \dd t,
\]
so since $A_0$ has a bounded $H^\infty$-calculus on $L^q$, see
Proposition~\ref{p-semigr}~\ref{p-semigr:ii}, we have the equality of operators
\[ \int_{r}^\infty A_0 \e^{-t^2 A_0} \; \dd t = \frac{1}{r} \psi(r^2 A_0)
\]
in $L^q$ for every $1 < q < 2$. Thus, denoting $I_k := \{ j \in I : r_j
\vee 2^{-m} = 2^k \}$ for every $k \in \Z$, we get
\[ \sum_{j \in I} \int_{r_j \vee 2^{-m}}^\infty A_0 \e^{-t^2 A_0} b_j \; \dd t
	= \sum_{k \in \Z} \sum_{j \in I_k} \frac{1}{r_j \vee 2^{-m}}
	\psi \bigl( (r_j \vee 2^{-m})^2 A_0 \bigr) b_j = \sum_{k \in \Z}
	\psi(4^k A_0) \sum_{j \in I_k} \frac{b_j}{r_j \vee 2^{-m}}.
\]
After these preparations we actually start the estimate. Let $q := d/(d-1)$ be
the Sobolev conjugated index to $1$. Using the Tchebychev inequality for this
$q$, we get
\begin{align*}
  & \Bigl| \Bigl\{ \mathrm x \in \Omega : \Bigl| \sum_{j \in I}
	\int_{r_j \vee 2^{-m}}^\infty A_0 \e^{-t^2 A_0} b_j (\mathrm x)
	\; \dd t \Bigr| > \frac{\sqrt{\pi}\alpha}{8} \Bigr\} \Bigr| \\
  \le\strut& \frac{C}{\alpha^q} \Bigl\| \sum_{j \in I}
        \int_{r_j \vee 2^{-m}}^\infty A_0 \e^{-t^2 A_0} b_j \; \dd t
	\Bigr\|_{L^q}^q = \frac{C}{\alpha^q} \Bigl\| \sum_{k \in \Z}
	\psi(4^k A_0) \sum_{j \in I_k} \frac{b_j}{r_j \vee 2^{-m}}
	\Bigr\|_{L^q}^q.
%
\intertext{Observe, that the sum over $k$ is in fact a finite sum, since $I_k$
is empty for $k < -m$ by definition and for large $k$ by the finite measure of
$E$, cf.\@ \eqref{e-|E|}. Thus, there is no convergence problem in applying
Lemma~\ref{RHinftylemma}, which helps to estimate this expression further by}
\le\strut & \frac{C}{\alpha^q} \biggl\| \Bigl( \sum_{k \in \Z} \Bigl|
        \sum_{j \in I_k} \frac{b_j}{r_j \vee 2^{-m}} \Bigr|^2 \Bigr)^{1/2}
	\biggr\|_{L^q}^q =
	\frac{C}{\alpha^q} \int_\Omega \Bigl( \sum_{k \in \Z} \Bigl|
	\sum_{j \in I_k} \frac{b_j(\mathrm x)}{r_j \vee 2^{-m}} \Bigr|^2
	\Bigr)^{q/2} \; \dd \mathrm x.
\intertext{Now, by \ref{CZZ(5)} the sum over $k$ is finite for every $\mathrm x
\in \Omega$ and the number of addends is even bounded uniformly in $\mathrm x$
and in $m$, so by the equivalence of norms in finite dimensional spaces, we may
continue to estimate by}
\le \strut & \frac{C}{\alpha^q} \int_\Omega \Bigl( \sum_{k \in \Z} \Bigl|
        \sum_{j \in I_k} \frac{b_j(\mathrm x)}{r_j \vee 2^{-m}} \Bigr| \Bigr)^q
	\; \dd \mathrm x \le \frac{C}{\alpha^q} \int_\Omega \Bigl(
	\sum_{j \in I} \frac{|b_j(\mathrm x)|}{r_j \vee 2^{-m}} \Bigr)^q
	\; \dd \mathrm x.
\intertext{Next we estimate $r_j \vee 2^{-m}$ by $r_j$ and, using again the
equivalence of norms in the finite sum over $j$, we get}
\le \strut &\frac{C}{\alpha^q} \int_\Omega \sum_{j \in I}
	\frac{|b_j(\mathrm x)|^q}{r_j^q} \; \dd \mathrm x \le
	\frac{C}{\alpha^q} \sum_{j \in I} \ell_j^{- q} \int_\Omega \bigl|
	b_j(\mathrm x) \bigr|^q \; \dd \mathrm x,
\end{align*}
since $r_j \sim \ell_j$. Using once more the Sobolev embedding $W^{1,1}
\hookrightarrow L^{d/(d-1)} = L^q$, we see as in \eqref{PoincSobVerweis}
\[ \int_\Omega \bigl| b_j(x) \bigr|^q \; \dd x = \| b_j \|_{L^q}^q \le C \bigl(
	\alpha \ell_j^d \bigr)^q = C \alpha^q \ell_j^{dq}.
\]
Summarizing we have shown
\begin{align*}
  \Bigl| \Bigl\{ \mathrm x \in \Omega : \Bigl| \sum_{j \in I}
	\int_{r_j \vee 2^{-m}}^\infty A_0 \e^{-t^2 A_0} b_j (\mathrm x)
	\; \dd t \Bigr| > \frac{\sqrt{\pi}\alpha}{8} \Bigr\} \Bigr| &\le
	\frac{C}{\alpha^q} \sum_{j \in I} \ell_j^{-q} \alpha^q \ell_j^{dq} = C
	\sum_{j \in I} \ell_j^d\\
  &\le C \sum_{j \in I} |Q_j| \le \frac{C}{\alpha^p} \| f\|_{W^{1,p}_D}^p,
\end{align*}
using one final time \ref{CZZ(4)}.
\end{proof}
It remains to prove Lemma~\ref{RHinftylemma}, which serves as a substitute for
Lemma~4.14 in \cite{auschmem}. We give a different proof, that instead of
$L^p$-$L^2$ off-diagonal estimates relies on the $H^\infty$ functional calculus
of the operator and gives the assertion for the full range of $1 < q < \infty$.
%
\begin{lemma} \label{RHinftylemma}
Let $1 < q < \infty$, let $-B$ be the generator of a bounded analytic semigroup
on $L^q$, such that $B$ and $B'$ admit bounded $H^\infty$-calculi on
$L^q$ and $L^{q'}$, respectively and let $\psi \in H^\infty_0(\Sigma_\phi)$ for
some $\phi \in \left] \varphi_B^\infty, \pi \right]$, where $\varphi_B^\infty$
is the $H^\infty$-angle of $B$. Then for every choice of functions $f_k \in
L^q$, $k \in \Z$, we have
\[ \Bigl\| \sum_{k \in \Z} \psi(4^k B) f_k \Bigr\|_{L^q} \le C \Bigl\| \Bigl(
	\sum_{k \in \Z} |f_k|^2 \Bigr) ^{1/2} \Bigr\|_{L^q},
\]
whenever the left hand side is convergent.
\end{lemma}
%
Before starting the proof, we observe, that thanks to \cite[Theorem~5.3]{KW01},
the operator $B$ even has an $\cR$-bounded $H^\infty$-calculus of angle
$\varphi_B^\infty$ on $L^q$, which means, that for every $\phi >
\varphi_B^\infty$ and every bounded set of functions $\Xi \subseteq
H^\infty(\Sigma_\phi)$ the set of operators $\{ \xi(A) : \xi \in \Xi \}$ is
$\cR$-bounded in $\cL(L^q)$. Here a set $\cT \subseteq \cL(L^q)$ is called
$\cR$-bounded, if there is a constant $C \ge 0$, such that for every $N \in
\N$, for every choice of functions $f_k \in L^q$, $k=1,\dots, N$, operators
$T_k \in \cT$, $k = 1, \dots, N$, and $\{-1,1\}$-valued, symmetric and
independent random variables $\epsilon_k$, $k=1, \dots, N$, on some probability
space $S$, we have
\[ \Bigl\| \sum_{k=1}^N \epsilon_k T_k f_k \Bigr\|_{L^2(S; L^q)} \le C \Bigl\|
	\sum_{k=1}^N \epsilon_k f_k \Bigr\|_{L^2(S; L^q)}.
\]
In the proof of Lemma~\ref{RHinftylemma}, we will use the following Lemma from
\cite[Lemma~4.1]{KW01} (see also \cite{DDHPV04}).
%
\begin{lemma}\label{KaltonWeisLemma}
Let $1 < q < \infty$, let $-B$ be the generator of a bounded analytic semigroup
on $L^q$, such that $B$ admits a bounded $H^\infty$-calculus on $L^q$ and let
$\psi \in H^\infty_0(\Sigma_\phi)$ for some $\phi \in {] \varphi_B^\infty,
\pi]}$. Then there is a constant $C \ge 0$, such that for every bounded
sequence $(\alpha_k)_{k \in \Z} \subseteq \C$ and every $t > 0$ we have
\[ \Bigl\| \sum_{k \in \Z} \alpha_k \psi(2^k t B) \Bigr\|_{\cL(L^q)} \le C
        \sup_{k \in \Z} |\alpha_k|.
\]
\end{lemma}
%
%
\begin{proof}[Proof of Lemma~\ref{RHinftylemma}]
Since $\psi \in H^\infty_0(\Sigma_\phi)$, there exists an $\epsilon > 0$ with 
$|\psi(z)| \le C |z|^\epsilon/(1 + |z|)^{2\epsilon}$ for all $z \in
\Sigma_\phi$. Let $\delta \in {]0, \epsilon[}$ and set
\[ \psi_1(z) := \frac{z^\delta}{(1 + z)^{2\delta}}, \quad \psi_2(z) :=
        \frac{(1 + z)^{2 \delta}}{z^\delta} \psi(z), \quad z \in \Sigma_\phi.
\]
Then we have $\psi_1, \psi_2 \in H^\infty_0(\Sigma_\phi)$, $\psi = \psi_1
\psi_2$ and $(\psi_1(B))' = \overline{\psi_1}(B')$.

Now, let $N \in \N$ and let $g \in L^{q'}$ with $\|g\|_{L^{q'}} = 1$, where
$1/q + 1/q' = 1$. Then for every family of $\{-1,1\}$-valued, symmetric and
independent random variables $\epsilon_k$, $k = -N, \dots, N$, on some
probability space $S$, we have
\[ \Bigl| \int_\Omega \sum_{k=-N}^N \bigl( \psi(4^k B) f_k \bigr) (x) g(x)
        \; \dd x \Bigr| = \Bigl| \int_S \sum_{k=-N}^N \epsilon_k^2(\sigma)
	\int_\Omega \bigl( \psi_2(4^k B) f_k \bigr) (x) \bigl(
	\overline{\psi_1} (4^k B') g \bigr) (x) \; \dd x \; \dd \sigma \Bigr|.
\]
Since the random variables $\epsilon_k$, $k = -N, \dots, N$, are independent
and thus orthogonal in $L^2(S)$, we may write this as
\begin{align*}
  &= \Bigl| \int_S \sum_{j,k = -N}^N \epsilon_k(\sigma) \epsilon_j(\sigma)
	\int_\Omega \bigl( \psi_2(4^k B) f_k \bigr) (x) \bigl(
	\overline{\psi_1} (4^j B') g \bigr) (x) \; \dd x \; \dd \sigma \Bigr|
	\\
  &\le \int_S \Bigl| \int_\Omega \sum_{k = -N}^N \epsilon_k(\sigma) \bigl(
	\psi_2(4^k B) f_k \bigr) (x) \sum_{j = -N}^N \epsilon_j(\sigma) \bigl(
	\overline{\psi_1} (4^j B') g \bigr) (x) \; \dd x \Bigr| \; \dd \sigma
  \intertext{and using twice the H\"older inequality we estimate by}
  &\le C \Bigl\| \sum_{k = - N}^N \epsilon_k \psi_2 (4^k B) f_k
	\Bigr\|_{L^2(S; L^q)} \Bigl\| \sum_{j = - N}^N \epsilon_j
        \overline{\psi_1} (4^j B') g \Bigr\|_{L^2(S; L^{q'})}.
\end{align*}
Now, in the first factor we use the $\cR$-bounded $H^\infty$-calculus of $B$.
Since the set of functions $\{ \psi_2(4^k \cdot) : k \in \Z \}$ is bounded in
$H^\infty(\Sigma_\phi)$, we get

\[ \Bigl\| \sum_{k= - N}^N \epsilon_k \psi_2(4^k B) f_k \Bigr\|_{L^2(S; L^q)}
	\le C \Bigl\| \sum_{k = - N}^N \epsilon_k f_k \Bigr\|_{L^2(S; L^q)} \le
	C \Bigl\| \Bigl( \sum_{k = - N}^N |f_k|^2 \Bigr)^{1/2} \Bigr\|_{L^q},
\]
where the last inequality follows from Khinchin's inequality
(cf.\@ \cite[1.10]{DJT95}).

In order to estimate the second factor, we apply Lemma~\ref{KaltonWeisLemma}
and get
\begin{align*}
  \Bigl\| \sum_{j = - N}^N \epsilon_j \overline{\psi_1} (4^j B') g
        \Bigr\|_{L^2(S; L^{q'})} &\le \Bigl( \int_S \Bigl\| \sum_{j= - N}^N
	\epsilon_j(\sigma) \overline{\psi_1} (2^{2j} B')
	\Bigr\|_{\cL(L^{q'})}^2 \|g\|_{L^{q'}}^2 \; \dd \sigma \Bigr)^{1/2} \\
  &\le \Bigl( \int_S \bigl( \sup_{j = -N}^N |\epsilon_j(\sigma)| \bigr)^2
        \; \dd \sigma \Bigr)^{1/2} = 1.
\end{align*}
This implies
\[ \Bigl\| \sum_{k = - N}^N \psi(4^k B) f_k \Bigr\|_{L^q} =
        \sup_{g \in L^{q'};\|g\|_{L^{q'}} = 1} \Bigl| \int_\Omega
        \sum_{k = - N}^N \bigl( \psi(4^k B) f_k \bigr) (x) g(x) \; \dd x
        \Bigr| \le C \Bigl\| \Bigl( \sum_{k = - N}^N |f_k|^2 \Bigr)^{1/2}
        \Bigr\|_{L^q}
\]
for every $N \in \N$. Letting $N \to \infty$ the assertion follows.
\end{proof}
Let us now come to the final step of the proof of the second assertion of
Theorem~\ref{t-mainsect}. Inequality~\eqref{e-weak} can be interpreted
as follows: $A_0^{1/2}$ is a continuous operator from $C^\infty_D(\Omega)$ --
equipped with the $W^{1,p}$-norm -- into the Lorentz space $L_{p,\infty}$, cf. 
\cite[Ch. 1.18.6]{triebel}. The space $L_{p,\infty}$ is identical (as a set)
with $(L^\infty,L^1)_{\frac 1p,\infty}$, and its quasinorm
$f \mapsto \sup_{t \ge 0} t^p |\{\mathrm x : |f(\mathrm x)| > t\} |$ 
is equivalent to the $(L^\infty,L^1)_{\frac 1p,\infty}$-norm 
(see \cite[Ch.~1.18.6]{triebel}); i.e. under a suitable renorming $L_{p,\infty}$
is an ordinary Banach space. Hence, $A_0^{1/2}$ uniquely extends by
density to a continuous operator from $W^{1,p}_D$ into $L_{p,\infty}$.
Thus, up to now, we have the two continuous mappings
\[ A_0^{1/2} : W^{1,2}_D \to L^2 
\]
and 
\[ A_0^{1/2} : W^{1,p}_D \to L_{p,\infty}
\]
for all $1 < p < 2$. Let $q \in {]1,2[}$ and choose $p \in {]1,q[}$. Using real
interpolation, this gives the continuous mapping
\[ A_0^{1/2} : (W^{1,p}_D, W^{1,2}_D)_{\theta,q} \to (L_{p,\infty},L^2)_{\theta,q}.
\]
Setting $\theta = \frac 2q \frac{q - p}{2 - p}$, the left hand side is equal to
$W^{1,q}_D$ by Theorem~\ref{interp} and the right hand side equals $L^q$
according to \cite[Thm.~2 Ch.~18.6]{triebel}.
This finishes the proof.
%
\begin{corollary} \label{c-defberfrac}
 Under the above assumptions, one has for $p \in {]1,2]}$ and $\beta \in
 {]0,\frac 12[}$
 \begin{equation} \label{e-defberfrac}
  \mathrm{dom}_{L^p}(A_0^\beta) = [L^p,W^{1,p}_D]_{2\beta}.
\end{equation}
\end{corollary}
%
\begin{proof}
 The operator $A_0$ admits bounded imaginary powers, according to
 Proposition~\ref{p-semigr}~\ref{p-semigr:ii}. Hence, \eqref{e-defberfrac}
 follows from a classical result, see \cite[Ch.~1.15.3]{triebel}.
\end{proof}

\begin{remark} \label{r-interpol00}
 In view of this result it would be highly interesting to determine also the
 interpolation spaces in formula~\eqref{e-defberfrac}. We suggest the formula
 \begin{equation} \label{e-sugg}
  [L^p,W^{1,p}_D]_{\theta} = \begin{cases}
	H^{\theta,p}, \; \text{if} \; \theta < \frac {1}{p} \\
	H_D^{\theta,p}, \; \text{if} \; \theta > \frac {1}{p},
	\end{cases}
 \end{equation}
 $H^{\theta,p}$ being the space of Bessel potentials and $H_D^{\theta,p}$ being
 the subspace which is defined via the trace-zero condition on $D$.
 Unfortunately, we are not able to prove this at present; but in the more
 restricted context of so called regular sets \eqref{e-sugg} is shown in
 \cite{ggkr}. Compare also \cite[Section~5]{HMRS09} for a simple
 characterization of regular sets in case of space-dimensions 2 and 3, and see
 also \cite{mitrea}.
\end{remark}
%
%
%
%
%
%
%
%
\section{Consequences} \label{sec-Cons}
%
%
%
%
%
%
%
%
\noindent
In this section we come back to the original motivation of our work, namely to
carry over  results which are known for divergence operators, when acting on
$L^p$ spaces, to the spaces from the scale $W^{-1,q}_D$, $ q \in {[2,\infty[}$,
compare also \cite{ausch/tcha01}, \cite[Section~5]{e/r/s}, \cite{hal/re0},
\cite{hal/re1}. In particular, this affects maximal parabolic regularity, which
is an extremely powerful tool for the treatment of linear and nonlinear
parabolic equations with nonsmooth data, see e.g.\@ \cite{pruess} or
\cite{hal/re0}. The crucial point is that this allows to treat
a discontinuous time-dependence of the right hand side,
which is relevant for applications. Moreover, the spaces
$W^{-1,q}_D$ allow to include distributional
right hand sides; the reader may think, e.g.\@ of electric surface densities,
concentrated on interfaces between different materials -- even when these interfaces
move in time.

\begin{definition} \label{d-positiv}
Following \cite[Ch.1.14]{triebel}, we call a densely defined operator $B$ on a Banach space $X$
positive, if it satisfies the resolvent estimate 
\[ \|(B+\lambda)^{-1}\|_{\mathcal L(X)} \le \frac {c}{1+\lambda}
\]
for a constant $c$ and all $\lambda \in [0,\infty[$.
(Note that a positive operator is sectorial in the sense of \cite[Ch.~1.1]{DHP03}.)
\end{definition}

Let us recall the notion of maximal parabolic regularity.
\begin{definition} \label{d-maxreg}
Let $1 < s < \infty$, let $X$ be a Banach space and let $J := \left]T_0, T \right[ 
\subseteq \R$ be a bounded interval. Assume that $B$ is a closed
operator in $X$ with dense domain $ D$ (in the sequel always equipped with the graph norm).
We say that $B$ satisfies \emph {maximal parabolic $L^s(J; X)$ regularity}, if for any
$f\in L^s(J; X)$ there exists a unique function $u \in W^{1,s}(J; X) \cap
L^s(J; D)$ satisfying
\[  u' + Bu = f,\quad \quad u(T_0) = 0,
\]
where the time derivative is taken in the sense of $X$-valued distributions on
$J$ (see \cite[Ch~III.1]{amannbuch}).
\end{definition}
%
%
\begin{remark} \label{r-ebed}
\begin{enumerate}
\item It is well known that the property of maximal parabolic regularity of
        an operator $B$ is independent of $s \in \left] 1, \infty \right[$ and
        the specific choice of the interval $J$ (cf.\@ \cite{dore}). Thus, in
	the following we will say for short that $B$ admits maximal parabolic
        regularity on $X$.
\item If an operator satisfies maximal parabolic regularity on a Banach space
	$X$, then its negative generates an analytic semigroup on $X$
	(cf.\@ \cite{dore}). In particular, a suitable left half plane belongs
	to its resolvent set.
\end{enumerate}
\end{remark}
%
%

\begin{lemma} \label{l-0451}
 Let $X,Y$ be two Banach spaces, where $X$ continuously and densely injects
 into $Y$. Assume that $B$ is a positive operator on $X$, such that $B^\beta :
 X \to Y$ is a topological isomorphism for some $\beta \in {]0,1]}$. Then the
 following holds true.
 \begin{enumerate}
  \item $B$ admits an extension $\widetilde B$ on $Y$, which also is a positive
	operator there.
  \item If $B$ admits an $H^\infty$-calculus, then $\widetilde B$ admits an
	$H^\infty$-calculus with the same $H^\infty$-angle.
  \item If $B$ satisfies maximal parabolic regularity on $X$, then
	$\widetilde B$ satisfies maximal parabolic regularity on $Y$.
 \end{enumerate}
\end{lemma}
\begin{proof} 
The well-known Balakrishnan formula $B^{-\beta}=\frac {\sin \pi \beta}{\pi} \int_0^\infty t^{-\beta}
(B+t)^{-1} \; \dd t$ (see \cite[Ch.~2.6]{pazy})
shows that the resolvent commutes with the fractional power $B^{-\beta} $. Hence,
for $\psi \in X$ and $\lambda \ge 0$ one can estimate
\begin{align*}
  \| (B + \lambda)^{-1} \psi \|_Y &= \bigl\| B^\beta (B + \lambda)^{-1}
	B^{-\beta} \psi \bigr\|_Y \\
  &\le \| B^\beta \|_{\mathcal L(X;Y)} \| (B + \lambda)^{-1} \|_{\mathcal L(X)}
	\| B^{-\beta} \|_{\mathcal L(Y;X)} \| \psi \|_Y \\
  &\le \| B^\beta \|_{\mathcal L(X;Y)} \| B^{-\beta} \|_{\mathcal L(Y;X)}
	\frac{c}{1 + \lambda} \|\psi\|_Y.
\end{align*}
This shows that the resolvent of $B$ may be continuously extended to $Y$ and that this extension 
admits the estimate
$\|\widetilde {(B+\lambda)^{-1}}\|_{\mathcal L(Y)}\le \frac {\tilde c}{1+\lambda}$.
Thus, one defines the extension $\widetilde B$ of $B$ to $Y$ as the inverse of $\widetilde {B^{-1}}$.
Since $X \hookrightarrow Y$, $\mathrm{dom}_X(B) \hookrightarrow
\mathrm{dom}_Y(\widetilde B)$. But $\mathrm{dom}_X(B)$
is dense in $X$ by the definition of a positive operator and $X$ was dense in $Y$ by
our assumption. Thus, $\mathrm{dom}_Y(\widetilde B) \supset \mathrm{dom}_X(B)$
is also dense in $Y$.
For (ii) see \cite[Prop.~2.11]{DHP03}.
Finally, assertion (iii) is proved in \cite[Lemma~5.12]{hal/re1}. The main
idea is again that the parabolic solution operator on $L^r(J;X)$ commutes with
the fractional power $B^{-\beta}$.
\end{proof}
\begin{theorem} \label{t-carryover}
 Let $\Omega$ and $D$ satisfy the Assumption ~\ref{assu-general}, let $\mu$ satisfy
 Assumptions~\ref{assu-coeff} and \ref{assu-coeffi} and assume $q \in {[2,\infty[}$.
Then the extension of $-\nabla \cdot \mu \nabla + 1$ from $L^q$ to $W^{-1,q}_D$ 
(being identical with the restriction from $W^{-1,2}_D$) has the following properties:
 \begin{enumerate}
  \item \label{t-carryover:i} It induces a positive operator.
  \item \label{t-carryover:ii} It admits a bounded $H^\infty$-calculus with
	$H^\infty$-angle $\arctan \frac{\|\mu\|_{L^\infty}}{\mu_\bullet}$; in
	particular, it admits bounded imaginary powers.
  \item \label{t-carryover:iii} It satisfies maximal parabolic regularity; in
	particular, its negative generates an analytic semigroup.
 \end{enumerate}
\end{theorem}
\begin{proof}
Thanks to Remark \ref{r-valid}, the transposed coefficient function $ \mu^T$
also satisfies Assumption \ref{assu-coeffi}. Hence, the operator
\begin{equation} \label{e-001} 
\bigl (-\nabla \cdot \mu^T \nabla +1 \bigr )^{1/2}:W^{1,p}_D \to L^p
\end{equation}
 provides a topological isomorphism for all $p \in {]1,2]}$, according to Theorem \ref{t-mainsect}.
Clearly, the adjoint operator of \eqref{e-001}, being identical with the operator 
$\bigl (-\nabla \cdot \mu \nabla +1 \bigr )^{1/2}:L^q \to W^{-1,q}_D$, with
$q=\frac {p}{p-1} \in {[2,\infty[}$, is also a topological isomorphism. Consequently, we need to know
the asserted properties only on the spaces $L^{q}$ due to Lemma~\ref{l-0451}.

In order to see this for \ref{t-carryover:i}, it suffices to note that on every
space $L^q$, $1 < q < \infty$, the operator $-\nabla \cdot \mu \nabla$
generates a strongly continuous semigroup of contractions (see
Proposition~\ref{p-semigr}), hence, the operator admits the required resolvent
estimate by the Hille-Yosida theorem.

Assertion~\ref{t-carryover:ii} is discussed in Proposition~\ref{p-semigr} and,
concerning \ref{t-carryover:iii}, the contraction property of the semigroup on
all $L^q$ spaces, provides maximal parabolic regularity on these spaces due to
a deep result of Lamberton (see \cite{lambert}).
\end{proof}

%
%
%
%
%
%

\end{document}